%% file: numalgo_2026_LargeBlockSale.tex
\documentclass[a4paper,fleqn]{cas-sc}
\usepackage{float}
\usepackage{algorithmic}
\usepackage{algorithm}
\usepackage{tikz}
\usepackage[authoryear,longnamesfirst]{natbib}

\def\tsc#1{\csdef{#1}{\textsc{\lowercase{#1}}\xspace}}
\tsc{WGM}
\tsc{QE}

\definecolor{davBlue}{rgb}{0,0.1843,0.6549}
\definecolor{domPink}{rgb}{.9373,.2784,.4353}

\usepackage{ulem}

\tikzset{
	cmtDebutDN/.pic={
		\draw[very thick,dashed] (0,0) -- (6,0); 
		\draw[very thick] (8,0) node{DEBUT David};
		\draw[very thick,dashed] (10,0) -- (15,0);
	}
}

\tikzset{
	cmtDebutDQ/.pic={
		\draw[very thick,dashed] (0,0) -- (5,0); 
		\draw[very thick] (8,0) node{DEBUT Dominique};
		\draw[very thick,dashed] (11,0) -- (15,0);
	}
}

\tikzset{
	cmtFin/.pic={
		\draw[very thick,dashed] (0,0) -- (6,0); 
		\draw[very thick] (8,0) node{FIN Modif};
		\draw[very thick,dashed] (10,0) -- (15,0);
	}
}

\newtheorem{definition}{Definition}
\newtheorem{theorem}{Theorem}
\newtheorem{lemma}[theorem]{Lemma}
\newtheorem{corollary}[theorem]{Corollary}
\newproof{pf}{Proof}
\newdefinition{rmk}{Remark}
\def\RR{{\mathbb{R}}}
\def\NN{{\mathbb{N}}}

\def\CC{{\mathcal{C}}}
\def\DD{{\mathcal{D}}}

\def\EN{numerical experiments}
\def\BnB{\textit{Branch \& Bound}}

\def\PNS{$\big(P_{N,M}\big)$}
\def\BPNS{$\mathbf{\big(P_{N,M}\big)}$}
\def\PS{$\big(PS_{N,M}\big)$}
\def\PRC{$\big(PC_{N,M}\big)$}

\def\PNSnm{$\big(P_{n,m}\big)$}

\def\ub2{$\mbox{UB}_1$}
\def\LS{\textit{LocalSolver}}
\def\Bon{\textit{Bonmin}}
\def\Cou{\textit{Couenne}}
\def\Sci{\textit{Scip}}

\def\Bar{\textit{Baron}}
\def\pht{\phantom{ }}

\begin{document}
\let\WriteBookmarks\relax
\def\floatpagepagefraction{1}
\def\textpagefraction{.001}

\shorttitle{}    

\shortauthors{}  

\title [mode = title]{Non-convex integer programming: a deterministic solution method of the large block sale problem}

\tnotemark[1] 

\tnotetext[1]{An extended abstract of this paper appeared in CODIT IEEE 2022, ISBN 978-1-6654-9607-0, pages 146-151} 

\author[1]{David Nizard}[orcid=0000-0001-9731-7191]


\ead{david.nizard@universite-paris-saclay.fr}

\affiliation[1]{organization={Laboratoire Interdisciplinaire des Sciences du Num\'erique (LISN), Universit\'e Paris Saclay},
            addressline={1 rue Ren\'e Thom}, 
            city={Gif-Sur-Yvette},
            postcode={91190}, 
            country={France}}

\author[1]{Dominique Quadri}
\cormark[1]


\ead{dominique.quadri@universite-paris-saclay.fr}

\cortext[1]{Corresponding author: Dominique Quadri}


\begin{abstract}
We study a non-convex, nonlinear integer mathematical program with a linear constraint. This problem originates in finance and consists of selling a large predetermined quantity of a financial asset over a fixed time horizon, while maximizing the total proceeds. The selling price is modeled using a deterministic penalty function capturing the market's adverse reaction to the selling flow. From a complexity perspective, the problem is in NP. We first solve the problem exactly using dynamic programming and propose several heuristics that yield tight numerical lower bounds.\\We then reformulate the problem into an equivalent factorable form and derive a tight convex relaxation using McCormick inequalities. This relaxation is embedded within an adapted \BnB\pht algorithm to find the global optimum in finite time. Numerical experiments show that our methods outperform several widely used generic solvers (\LS, \Sci, \Bar, \Cou, and \Bon) for selected instances, but do not yield tight upper bounds, beyond small instances.
\end{abstract}



\maketitle

\section{\textsc{Introduction and problem formulation} }\label{sectionIntro}
	\input{numalgo2026_intro_v1.tex}

\section{\large \textsc{Related research and our contributions}}
\input{numalgo2026_littrev_contrib_v1.tex}

\section{\large \textsc{An exact solution method and heuristics for \BPNS}}\label{sectionLB}
	\input{numalgo2026_DP_intro_v1.tex}
	
	\subsection{An exact solution method based on Dynamic Programming}\label{sectionDPExact}
		\input{numalgo2026_exactDP_v1.tex} 

	\subsection{Two-step method based on dynamic programming (TSDP)}\label{sectionLB_TS}
		\input{numalgo2026_LB_TS_v1.tex}

	\subsection{Iterated Local search (ILS)}\label{sectionILS}
		\input{numalgo2026_LB_ILS_v1.tex}

	\subsection{Basic heuristics}\label{sectionFSLIS}
		\input{numalgo2026_LB_FSLIS_v1.tex}

\section{\textsc{Convexification and upper bounds for \BPNS}}\label{sectionUB}
	\input{numalgo2026_UB_intro_v1.tex}

	\subsection{First upper bound using monotonicity}\label{sectionUB_mono}
		\input{numalgo2026_UB_mono_v1.tex}

	\subsection{Convexification}\label{sectionUB_convex}
		\input{numalgo2026_UB_convex_v1.tex}

	\subsection{Adapted \BnB\pht algorithm}\label{sectionUB_bnb}
		\input{numalgo2026_UB_bnb_v1.tex}

\section{\textsc{Numerical Experiments}}\label{sectionNE}
	\input{numalgo2026_EN_intro_v1.tex}

	\subsection{Experimental setup}\label{sectionNEFrame}
		\input{numalgo2026_EN_framework_v1.tex}

	\subsection{Small and medium-sized instances}\label{sectionNESmall}
		\input{numalgo2026_EN_SmallMed_v2.tex}

	\subsection{Large and very large instances}\label{sectionNELarge}
		\input{numalgo2026_EN_Large_v1.tex}

\section{\textsc{Conclusion and perspectives}}\label{sectionCCL}
	\input{numalgo2026_ccl_v3.tex}
\section*{\textsc{Declarations}}
\begin{itemize}
	\item Funding: no funding was received to conduct this study.
	\item Conflict of interest: 
	the authors have no financial or proprietary interests in any material discussed in this article.
	
	\item Data availability : data sets generated during the current study are available from the corresponding author on reasonable request. 
	
\end{itemize}

\bibliographystyle{cas-model2-names}

\bibliography{ejor_biblio_1}


\appendix
\section{\textsc{Appendix: Proof of Bellman Equation (theorem \ref{thBell}, section \ref{sectionDPExact})}}\label{sectionBellmanProof}
	\input{numalgo2025_appendix_bellman.tex}	

\section{\textsc{Appendix: Proof that the Non-Separable Integer Optimization Problem \PNS is NP }}\label{sectionNPHardProof}
	\input{numalgo2026_appendix_np.tex}	

\section{\textsc{Appendix: Additional results tables}}\label{sectionExtraTables}
	\input{numalgo2025_appendix_extratables.tex}	

\section{\textsc{Appendix: Calibration of the penalty function $g$}}\label{sectionCalibration}
	\input{numalgo2026_appendix_calibration.tex}

\section{\textsc{Appendix: List of main notations}}\label{sectionNotations}
	\input{numalgo2025_appendix_notations.tex}

\printcredits



\end{document}

%% file: numalgo2026_intro_v1.tex
In this paper, we present a solution method for a non-convex, nonlinear integer mathematical program (MINLP) in a maximization setting with a linear equality constraint. This problem originates in finance and is commonly referred to as the Large Block Sale (LBS) problem.\\It consists of selling a large predetermined quantity of a financial asset (such as a stock) over a finite time horizon, discretized into $N$ time periods, while maximizing the total proceeds from the sale. At each trading time, the realized sale price is modeled by a penalty function that reflects the level of information available to market participants and their adverse reaction to the selling pressure, as they progressively show lower bid prices that undermine the liquidation process.\\
From the mathematical programming standpoint, this problem is particularly challenging. Theoretically, the problem is in NP, as shown in Appendix \ref{sectionNPHardProof}. In practice, non-convexity requires us to adapt classical methods (e.g., \BnB, \textit{cuts}). Moreover, since no general-purpose solution algorithm is known for this class of mathematical programs, solution methods are often problem-specific.
In this work, we explore two complementary approaches. First, we analyze the problem’s structure and show that it satisfies the Bellman equation, allowing for an exact solution using dynamic programming (DP). This method proves effective for small and medium-sized instances. For larger instances, we propose DP-based heuristics that provide high-quality lower bounds.\\
Second, we reformulate the problem into an equivalent factorable form, in the sense of McCormick, and compute a tight convex relaxation by applying his inequalities. The resulting relaxation is embedded in a customized \BnB\pht algorithm that guarantees convergence to the global optimum in finite time. 
Finally, in \EN, we benchmark our results against several widely used generic solvers: \LS\pht and \Bon, which yield local optima, as well as global solvers such as \Bar, \Cou\pht and \Sci.\vspace{0.2cm}\\We now introduce our problem, denoted \PNS, and the corresponding notation as follows:\\
\begin{equation}
	\big(P_{N,M}\big)\left\{
	\begin{array}{rl}\label{pbMain0}
		\max f(x)&=\displaystyle\sum_{i=1}^{N}\Big[p_i- c_i\cdot g\big(S_i\big)\Big]\cdot x_i\\
		s.c\;\; \vline &S_N =M\\
		\vline
		&\forall i,\,0\leq x_i\leq M\\
		\vline
		&\forall i,\, x_i\in \NN 	
	\end{array}
	\right.
\end{equation}

where $S_i=\displaystyle\sum_{k=1}^{i} x_k$ is defined as the partial sum of $x$ up to $i$.\\\PNS\pht is a non-separable integer problem with one equality constraint. 
$g$ is a real-valued function called the penalty function. Coefficients $p_i$ and $c_i$ are strictly positive real numbers. Variables $(x_1,\cdots,x_i,\cdots,x_N)$ are integers.  Let $\DD = \left\{ (x_1,\ldots,x_N)\in \NN^N, \phantom{1} S_N = M \right\}$ be the admissible set. \PNS\pht is well-defined because $\DD$ is finite.\vspace{0.2cm}\\
The objective function $f$ is bounded. Indeed, we note $\overline{p}=\displaystyle\max_{i}p_i$ and obtain $0\leq f(x)\leq M\,\overline{p}$, which is a basic upper bound, independent of $g$.\\
We assume that the function $g$ is sufficiently regular, $\CC^2$ in practice, and increasing. From an economic standpoint, an increasing penalty function is a standard assumption as it reflects the memory of market participants.\\From a mathematical standpoint: 
\begin{itemize} 
	\item DP solution method works with any well-defined penalty function, as the proof of Theorem \ref{thBell} in Appendix \ref{sectionBellmanProof} does not depend on $g$, nor does the complexity.
	\item However, additional assumptions on $g$ are required, in Section \ref{sectionUB_mono}, to obtain a monotonic upper bound. $g$ must be increasing and $x\mapsto x\,g(x)$ convex.
	\item In Section \ref{sectionUB_convex}, we construct the convex hull of $g$ using Newton’s method, which, in the general case, requires that $g''$ exist. Thus, twice differentiability provides a sufficient condition on $g$.
\end{itemize}

In addition, we assume that $g(0)=0$ and $\displaystyle\lim_{+ \infty} g = 1$.
There is no loss of generality, as, for any function $h$, $\CC^2$ strictly increasing on $\RR_+$ (or at least on $[0;M]$), one can define $g$ as : $g(x)= H\cdot\frac{h(x)-h(0)}{h(M)-h(0)}$, with $H \in\, ]0;1]$. Then, we have $g(0)=0$ and $g(M)=H$, as close to $1$ as intended.

%% file: numalgo2026_littrev_contrib_v1.tex
In the mathematical programming literature, although still an open topic, non-convex MINLPs have been studied more intensively over the last decade. 
For non-convex quadratic programs, one natural idea lies in reformulating the problem into a more tractable form.  The work of \citet{bill_2012} introduces a quadratic convex reformulation (QCR).
\citet{lamb_2021} introduces a sequence of positive semidefinite matrices, whose (convex) relaxation converges to the optimum of the initial (non-convex) problem. The key ideas here are reformulation and solution methods, either through linearization or convexification. The works of \citet{sher_1998} and \citet{cafi_2012} present reformulation and linearization techniques (RLT) and extend them to polynomial programming.\\ The search for the tightest convex relaxation naturally leads to the study of convex hulls. \citet{mccor_1976} and \citet{meyer_2004_2} characterize bilinear and trilinear terms, respectively. \citet{cafi_2010} shows how to regroup lower-degree monomials, whose convex hull is known, to obtain tighter relaxations for quadrilinear terms.\\For higher-degree terms, where to the best of our knowledge, no general characterization is known, \citet{rikun_1997} shows that the convex hull of the product $x_1\,\cdots x_k$ corresponds to the set of vertices of a polyhedron, when the $x_i$ are box-constrained.\\In the general (non-polynomial, non-convex) case, no general solution is known, nor has it been proven that one does not exist. 
The seminal work of \citet{mccor_1976} introduces an equivalent reformulation for a class of factorable programs that can be expressed as sums or products of univariate functions of the original or auxiliary variables, which are themselves factorable. The author characterizes the convex hull of the reformulated problem and proposes a spatial \BnB\pht algorithm based on a sequence of progressively tighter convex relaxations, obtained as the variable domains are successively reduced. Convergence of the method to the optimum of the original problem is established up to a prescribed precision. In practice, however, despite the existence of theoretical solutions for a broad class of engineering problems, this approach can be computationally expensive.
Another interesting convexification method is proposed by \citet{andro_1995} who adds a quadratic penalty controlled by the parameter $\alpha$ to the non-convex objective function and studies the conditions under which the resulting function becomes a convex under estimator (in a minimization context). They incorporate these relaxations into a so-called $\alpha-$\BnB\pht algorithm, which reaches the global optimum. Many other methods (see for instance, the \citet{burer_2012} survey) follow the same classical scheme: transform the initial problem into a sequence of easier problems (separable, factorable, etc.), obtain a tight convex relaxation for each, and extract a sub-sequence that converges to the global optimum of the initial problem. 
However, to the best of our knowledge, our problem has not been addressed in the current literature.\\Its difficulty lies in both the non-convexity of the objective function and, more specifically, in the structure of the penalty function. Indeed, at every stage, $g$ depends on the sum of previous decision variables $x_1,\cdots,x_i$.\\We refer to this as a long-term memory model, where any given $x_i$ influences all subsequent penalties $g(S_k),\,k\geq i$. 
In the literature, we also examined alternative approaches based on recursive decomposition. If the problem structure allows us to isolate a subset of decision variables, and if, crucially, we can reconstruct a global solution from the solution of the sub-problems, then we may find an algorithm that decomposes the initial problem into smaller sub-problems, solves them independently in some manner, and combines their solutions to obtain a global optimum.   
\\In his seminal work on dynamic programming, \citet{bell_1957} defines conditions under which problem functions can be decomposed. He shows how to fix the last decision variable $x_N$, solve the sub-problems with $N-1$ decision variables independently. He establishes a recursive relation, known as the Bellman equation, between the optimum of the $N$-variable initial problem and the optima of the $N-1$ sub-problems. This recursion reduces the problem to single-variable cases, which are easier to solve, and allows the optimal solution of the initial problem to be reconstructed. We apply the Bellman equation to \PNS\pht in Section \ref{sectionDPExact} and solve it using DP.\\
Other classical examples of structure-based solution methods include the block decomposition of linear problems, with coupling constraints (e.g., \citet{dantz_1961}), or with coupling variables (e.g., \citet{bend_1962}, extended to the non-linear case by \citet{geof_1972}). For a broader discussion on non-convex global optimization, we refer the reader to the works of \citet{horst_1995} and \citet{tawa_2004}.\vspace{0.0cm}\\
In addition, the Large Block Sale problem can be viewed as an unbounded nonlinear knapsack problem with a non-separable objective function induced by cumulative decision variables. Within this context, our approach follows the spatial \BnB\pht paradigm and develops a problem-specific factorable reformulation and convex relaxation that exploit this cumulative structure, leading to tighter bounds than generic reformulation schemes. 
\vspace{1em}\\
In the financial literature, the LBS problem, which belongs to the larger class of optimal liquidation portfolio problems, has also been studied from many different angles. A first approach focuses on the market price impact of a large block sale from a statistical standpoint. Empirical studies can be found in \citet{guth_1965, dann_1977,gemm_1996}.
A second approach focuses on specific markets dedicated to such transactions, known as \textit{upstairs market}, and their more recent electronic counterpart the \textit{dark pools}. The reader can refer, for instance, to \citet{madh_1997} and more recently \citet{mish_2017}, which address LBS in these markets.\\
A third approach is to model incentives of the market participants involved in a block transaction (block trader, broker, specialist, etc.), study the price equilibrium, and derive existence and uniqueness conditions. \citet{sepp_1990} and \citet{keim_1996} provide examples.\\
A fourth type of model, where price impact is inferred from the modeling of the Limit Order Book (LOB), which ranks best bid and ask prices and their corresponding volumes, was first proposed in \citet{obiz_2005}. Refinements in that direction can be found in \citet{alfo_2008,alfo_2010} and \citet{obiz_2013}. \\
Lastly, there are also abundant references on optimal liquidation in continuous time, where LBS turns into a stochastic control problem (see \citet{barle_1994}, \citet{vath_2006}, \citet{seydel_2009}, \citet{khar_2010}, \citet{gath_2010,gath_2012}).\\ 
However, in this article we consider LBS in discrete time, with integer variables. Our central problem, how to optimally split the block into smaller orders to maximize the sale proceeds (or equivalently minimize overall price impact) was previously investigated in the seminal work of \citet{bert_1998} and \citet{almg_2000,almg_2003}.  \\
They introduce a stochastic stock price, whose dynamics account for previous trade sizes. The majority of models distinguish between the \textit{temporary} price impact, due to the temporary imbalances between supply and demand at a given point in time, and the \textit{permanent} impact which reflects the effect on share price of the information conveyed by our trading (as defined in \citet{almg_2003}). Their solution approaches are in some cases based on DP. In most cases, however, related price impact functions obey linear or power laws.\\
The \textit{market memory}, or information conveyed to market participants during the selling process also plays a crucial role. Indeed, the majority of models feature either a \textit{short memory}, in the sense that their penalty at time $t_k$ is in $g(x_k)$, depending only on the $x_k$ units traded at that time, or their long term memory is separated (e.g \citet{almg_2000}), as the penalty at time $t_k$ is in $\sum_{j=1}^{k} g(x_j)$.\\
We differ from these models by assuming deterministic asset prices, which is more restrictive, but also by considering a long-term, non-separated memory model and a much larger class of penalty functions.\vspace{0.2cm}\\
Outside of finance, a closely related question arises in the Convoy Movement Problems: in order to transport a given quantity of supplies through enemy lines, within a fixed time horizon (or with a limited number of convoys), how should the optimal convoy size be determined? See \citet{falco_1976} and \citet{mcclo_1987} for practical examples from World War II.\\\\
\textbf{Main Contributions:}
We propose two exact solution methods, based either on dynamic programming or on non-convex, non-linear mathematical programming. The main contributions of our work are as follows:
\begin{itemize}
	\item We propose a long-term memory model, with a price impact function in $g(S_k)$, for non-linear bounded functions g, which is, to the best of our knowledge, new in the financial literature, related to the LBS problem. 
	\item We adapt DP to solve our non-convex MINLP exactly for small instances.	We also introduce a two-stage DP-based heuristic for medium and large instances that provides either the optimum or very tight lower bounds.
	
	\item We develop multiple heuristic methods for computing good lower bounds for \PNS. 
	\item Using factorization and McCormick inequalities, we compute a tight convex relaxation of \PNS\pht and incorporate it into an adapted \BnB\pht algorithm, which returns either the optimum or an upper bound.  
	\item In \EN\pht related to our financial application, we benchmark our results against relevant, widely used generic solvers and obtain better results for several instances.
\end{itemize}\vspace{0.2cm}
This article is organized as follows. Section \ref{sectionLB} introduces the Bellman equation and applies dynamic programming to solve the problem exactly. It also presents several heuristics for computing lower bounds, including DP-based methods and local search, along with a discussion of their complexity.\\Section \ref{sectionUB} focuses on upper bounds. Section \ref{sectionUB_mono} first introduces a naive upper bound deduced from the monotonicity of the penalty function. Then, Section \ref{sectionUB_convex} presents an equivalent factorized formulation of the problem and computes a convex relaxation using McCormick inequalities. Finally, Section \ref{sectionUB_bnb} shows how this tight relaxation can be used within a \BnB\pht algorithm to solve the initial problem.\\Section \ref{sectionNE} presents our \EN. We compare the performance of our proposed methods against several widely used generic solvers across a range of instances derived from our financial application. Section \ref{sectionCCL} concludes and outlines directions for future research.

%% file: numalgo2026_DP_intro_v1.tex
In this section, we exploit the problem structure and apply dynamic programming (DP) to either solve \PNS\pht exactly in Section \ref{sectionDPExact}, or to derive very tight lower bounds in Section \ref{sectionLB_TS}. Lastly, in Section \ref{sectionILS}, we introduce additional heuristics based on a local search algorithm.

%% file: numalgo2026_exactDP_v1.tex
We begin by proving Bellman's equation and describing the DP algorithm, introduced by \citet{bell_1957}, which is used here to obtain the exact solution.\PNSnm\pht is a discrete problem with bounded decision variables, for $n\leq m$.\\Let $O_{n,m}$ denote its optimal value. The following result can be proven directly by induction on $n$ (see the proof in Appendix \ref{sectionBellmanProof}):\\
\begin{theorem}\label{thBell}
	$\forall n,m,\;1\leq n\leq m,\;O_{n,m}=\displaystyle\max_{i \in [\![0;m]\!]} \left\{ O_{n-1,m-i}  + \big[p_n-c_n\,g\big(m\big)\big]i\;\right\}$,\\ with
	initial conditions $O_{n,0}=O_{0,m}=0$. \vspace{0.0cm} 
\end{theorem}
This theorem leads to the following recursive solution scheme:
\begin{algorithm}[H]
	\caption{Exact Dynamic Programming}
	\begin{algorithmic}
		\STATE \textit{\textbackslash\textbackslash Build $\big(O_{n,m}\big)$ matrix}
		\FOR {$n=1$ \TO $N$}
		\FOR {$m=1$ \TO $M$}
		\FOR {$k=0$ \TO $m$}
		\STATE $O_{n,m}= \max\Big(O_{n,m},O_{n-1,m-k}+ \big[p_n-c_n\,g\big(m\big)\big]\,k\;\Big)$ 
		\ENDFOR
		\ENDFOR
		\ENDFOR
		\STATE\textit{\textbackslash\textbackslash Compute optimal strategy by backtracking}
		\STATE $m= M$
		\FOR{$i=N$ \TO 1}
		\IF{$O_{i,m}=O_{i-1,m}$} 
		\STATE	$x_{i} = 0$
		\ELSE
		\STATE	$k= m$
		\WHILE{\Big($O_{i,m}\neq O_{n-1,m-k}+\big[p_n-c_n\,g\big(m\big)\big]\,k\Big)\;\land\;(k>0)$}
		\STATE $k= k-1$
		\ENDWHILE
		\STATE	$x_i= k$
		\STATE	$m= m-k$
		\ENDIF
		\label{algoDPExact}
		\ENDFOR
	\end{algorithmic}
\end{algorithm}
The space (memory) complexity is $O(NM)$. Filling the optimal matrix requires $O(NM^2)$ operations, while backtracking is $O(NM)$. Thus, the overall time complexity is $O(NM^2)$. Since $M$ is an integer input, the input size is  $O(N+\log M)$. Consequently, the running time is polynomial in the numerical value of $M$, but not in the size of its binary encoding, unless $M$ is bounded by a polynomial in $N$. Therefore, the time complexity of the exact DP is pseudo-polynomial. 
\vspace{0.2cm}\\
From a practitioner's standpoint, the typical float (number of shares available for sale in the market) of large companies traded on the NYSE or NASDAQ is of the order of magnitude of $10^8$ to $10^9$. Holding about $0.1$ to $1\%$ of the float, which is a significant block, translates into $M$ in the $10^5\,-\,10^8$ range.\\
The number of trading times controlled by $N$ is more subjective as it depends on the total liquidation horizon (a week to a month in practice) and on the block size relative to the average daily trading volume. For a typical blue-chip stock, liquidating a block over a month (18 trading days), 1 to 5 times a day, yields a range of $10$ to $100$ for $N$.\\
For instance, for \textit{Home Depot} (HD), a component of the Dow Jones, with a float of one billion shares, a block of 5 million shares, and an average daily volume of 4 million shares, liquidation over a week three times a day yields $(N,M)=(15,5.10^6)$, while liquidation over a month yields $(N,M)=(54,5.10^6)$\vspace{0.2cm}\\
Hence, DP remains tractable up to $N=10^2, M=10^5$. Table \ref{tabCpuCompLB} in Appendix \ref{sectionExtraTables} compares the CPU times of DP, TSDP and ILS.

%% file: numalgo2026_LB_TS_v1.tex
As the time complexity of exact DP grows as $N\,M^2$, it becomes too costly for large instances. Therefore, in this section, we present different heuristics to compute tight lower bounds within a short CPU time. We introduce our adapted DP-based technique, which we refer to as the two-step dynamic programming method (TSDP). It consists of two independent steps: a coarse-grained DP and a DP with bounds, described below: 
\begin{itemize} 
	\item \textbf{Coarse-grained DP}: when $M$ is large, exact DP is not efficient. A natural idea is to sell the asset in buckets of size $P$  $(P\gg 1)$, which we refer to as grain $P$. The solution algorithm is similar to that of DP, with a new quantity $M'=\lfloor M/P\rfloor$. Hence, the time complexity becomes $O\left(\dfrac{N\,M^2}{P^2}\right)$, which is faster by a factor of $P^2$. However, the coarser the grain, the faster the heuristic, but the lower its quality. There is a trade-off, which is analyzed in Section \ref{sectionNE_Small_DP}. 
	\item \textbf{DP with bounds}: We refine the heuristic obtained in the previous step by applying the exact DP algorithm in its neighborhood. We impose bounds on the admissible solutions in order to restrict the search space. 
\end{itemize}
\begin{algorithm}[H]
	\caption{Dynamic Programming with Bounds}
	\begin{algorithmic}
		\STATE\textit{\textbackslash\textbackslash Build $\big(O_{n,m}\big)$ matrix}
		\FOR{$n=1$ \TO $N$}
		\FOR{$m=L_n$ \TO $U_n$}
		\FOR{$k=l_n$ \TO $u_n$}
		\STATE $O_{n,m}= \max\Big(O_{n,m},O_{n-1,m-k}+ \big[p_n-c_n\,g\big(m\big)\big]\,k\;\Big)$\; 
		\ENDFOR
		\ENDFOR
		\ENDFOR
		\STATE \textit{\textbackslash\textbackslash Backtracking is identical to Algorithm (\ref{algoDPExact})}
		\STATE\textit{\textbackslash\textbackslash Get solution by backtracking in $O_{N,M}$ matrix}
		\label{algoDPBounds}
	\end{algorithmic}
\end{algorithm}
We now introduce the following definitions to compute the complexity:\\
Let $x^0=\left(x_1^0,\cdots,x_N^0\right)$ be an initial solution.\\
Let $l$ and $u$ denote the lower (resp. upper) bounds for $x$ such that: $\forall\, i\; 0 \leq l_i\leq x_i\leq u_i\leq M$\\
Let: $L_i= \min\left(\sum_{k=1}^{i} l_k, M\right)$ and $U_i=\min\left(\sum_{k=1}^{i}u_k,M\right)$. Hence, $\forall\,i\; L_i\leq S_i\leq U_i$.\vspace{0.2cm}\\ 
\textbf{Space complexity:} one has to store $O_{n,m}$ with $L_n\leq m\leq U_n$ and $1\leq n\leq N$. Hence, the complexity is $O\left(\sum_{n=1}^N U_n-L_n\right)$. If we assume that $u_n-l_n$ is bounded for all $n$ by an integer $R\leq M$, then $R$ represents the search neighborhood around a given heuristic solution. $R$ can always be defined as $R\leq M$.\\
$\forall n, U_n-L_n = \sum_{i=1}^n u_i-l_i\leq nR$, which yields a bounded space complexity of $O\left(N^2 R\right)$.\vspace{0.2cm}\\
\textbf{Time complexity:} an additional loop over $k$ is required to compute $O_{n,m}$, with $l_n\leq k\leq u_n$.\\
We obtain $\sum_{n=1}^N (U_n-L_n)(u_n-l_n)\leq R\sum_{n=1}^N (U_n-L_n)\leq R^2\sum_{n=1}^N n\leq \frac{1}{2}N(N+1)R^2$, which yields a time complexity of $O\left(N^2\,R^2\right)$. Since $R\leq M$, this complexity is bounded by $O(N^2M^2)$. Therefore, by the same argument as above, the TSDP time complexity is pseudo-polynomial.
\vspace{0.2cm}\\
Finally, the two steps in TSDP can either work synchronously, when the grain and the funnel size are controlled by the same parameter $P$, or independently, when a promising heuristic is refined through DP with bounds. However, unlike exact DP, this method no longer guarantees finding the global optimum.

%% file: numalgo2026_LB_ILS_v1.tex
DP-based solutions also face memory constraints. Since they require storing $N$ by $M$ matrices, they are not efficient for large instances. Therefore, we conclude this section by introducing a local search algorithm, which starts from an admissible solution, and returns a local maximum, while only storing the current decision variables and the objective value. 
Let $x^0=(x_1^0,x_2^0,\cdots,x_N^0)$ be an element of $\DD$. We first shift $x_1^0$ by $+P$ and $x_2^0$ by $-P$, for some fixed integer $P$, provided we stay in the feasible $[\![0;M]\!]$ domain. The solution remains admissible under the $P$-shift operator, as the sum of $x_i^0$ remains equal to $M$. We then apply the opposite shift $-P$, and so on for all pairs $(x_i,x_{i+1})$. If no transformation among these $2\,(N-1)$ possible shifts improves the objective function value, then the algorithm stops and returns the current solution $x^0$. Otherwise, the best solution found is saved as $x^1$ and the process is iterated. Therefore, the method returns a fixed point $x^{K+1}=x^K$ for the $P$-shift operator, which corresponds to a local maximum.\\
Lastly, we can apply a sequence of $P'$-shifts, with $(P'\neq P)$. For instance, we can perform a dichotomy on $P$: $P_0=2^R,\,P_1=2^{R-1},\cdots, P_R=1$, where $R=\lfloor\frac{\ln(M)}{\ln(2)}\rfloor$ is the largest integer such that a $2^R$ shift may be feasible.\\
\begin{algorithm}[H]
	\caption{Iterated Local Search}
	\begin{algorithmic}
		\STATE $runSum= 0$
		\STATE	$auxSum= 0$
		\STATE  $\Delta= 0$	
		\STATE	$f_{max}= 0$	
		\STATE	$i_{opt}= -1$
		\FOR{$i=1$ \TO N-1}
		\STATE		$d_+= 0$
		\STATE		$d_-= 0$
		\STATE		$runSum= runSum+x_i$
		\STATE		$auxSum= runSum+x_{i+1}$
		\STATE		$d_0=\big[p_i-c_i\,g(runSum)\big]\,x_i +\big[p_{i+1}-c_{i+1}\,g(auxSum)\big]\,x_{i+1}$
		\IF{($x_i+P\leq M) \land (x_{i+1}-P\geq 0)$}
		\STATE $d_+=\big[p_i-c_i\,g(runSum+P)\big]\,(x_i+P)+\big[p_{i+1}-c_{i+1}\,g(auxSum)\big]\,(x_{i+1}-P)$
		\ENDIF
		\IF{($x_{i+1}+P\leq M) \land (x_{i}-P\geq 0)$}
		\STATE $d_-=\big[p_i-c_i\,g(runSum-P)\big]\,(x_i-P)+\big[p_{i+1}-c_{i+1}\,g(auxSum)\big]\,(x_{i+1}+P)$
		\ENDIF
		\IF{$(d_+\geq d_-)$} 
		\IF{$(d_+-d_0 > \Delta)$}
		\STATE $\Delta = d_+-d_0$
		\STATE	$i_{opt}= i$
		\STATE	$\epsilon= 1$
		\ENDIF
		\ELSIF{$(d_--d_0 >\Delta)$} 	 
		\STATE $\Delta = d_--d_0$
		\STATE $i_{opt}= i$
		\STATE $\epsilon= -1$
		\ENDIF
		\IF{$i_{opt}>-1$} 
		\STATE			$x_{i_{opt}}= x_{i_{opt}} + \epsilon\cdot P$
		\STATE $x_{i_{opt}+1}= x_{i_{opt}} - \epsilon\cdot P$
		\STATE $f_{max}= f_{max} + \Delta$
		\ENDIF
		\ENDFOR	
		\label{algoILS}		
	\end{algorithmic}
\end{algorithm}
While ILS triggers a call to the objective function for every shift, we improve its efficiency by comparing only the relevant terms of $f$. Specifically, for every $i$, we increment the running sums $S_i$ and $S_{i+1}$ and compare only the affected terms of the objective function f, namely $d_0$ (no shift), $d_+$ (shift $(i,i+1,P)$) and $d_-$ (shift $(i,i+1,-P)$).
We can apply ILS either as a first-step heuristic, starting from a basic solution (cf. Section \ref{sectionFSLIS}), or as a second step to improve an existing local maximum.\vspace{0.2cm}\\
We provide an illustrative example. For $N=3,M=10,P=1$, we assume that, for instance, $(3,3,4)$ is an admissible solution and $f(3,3,4)$ is its objective value.\\ 
We start by computing the following six triplets $(2,4,4)$, $(2,3,5)$, $(3,2,5)$, $(4,2,4)$, $(4,3,3)$, $(3,4,3)$ and their respective objective values. If the maximum value is $f(3,3,4)$, the algorithm stops and returns $(3,3,4)$. Otherwise, the selected triplet, for instance $(2,3,5)$, becomes the new starting point, and the algorithm repeats by computing the next six triplets $(1,4,5)$, etc., until it reaches a fixed point.\\\\
We also notice that, when shifting consecutive terms $x_i$ by $x_i+P$ and $x_{i+1}$ by $x_{i+1}-P$, denoted $(i,i+1,P)$-shift, all the terms of $f$ involving $x_j$ $j\leq i-1$ are unchanged. Moreover, as $\pm P$-shifts cancel out, $x_i+x_{i+1}$ remains unchanged, and all the terms of order $j\geq i+2$ do not change either. Therefore the only difference in $f$ comes from the terms $d_+=\big[p_i-c_i\,g(S_i+P)\big]\,(x_i+P)+\big[p_{i+1}-c_{i+1}\,g(S_{i+1})\big]\,(x_{i+1}-P)$. Consequently, comparing $d+$ to its initial value $d_0=\big[p_i-c_i\,g(S_i)\big]\,x_i+\big[p_{i+1}-c_{i+1}\,g(S_{i+1})\big]\,x_{i+1}$ is enough to conclude whether $f$ is improved by the $(i,i+1,P)$-shift. Algorithm \ref{algoILS} simply compares every $(i,i+1,\pm P)$ shift this way, selects the maximum and iterates the process.

%% file: numalgo2026_LB_FSLIS_v1.tex
There are many ways to choose $N$ integers whose sum is $M$. We conclude this section by mentioning two intuitive (though naive) heuristics, which we use as benchmarks in numerical experiments in Section \ref{sectionNE}: 
\begin{itemize}
	\item \textbf{Fire sale} (abbreviated \textbf{FS}): the entire block is liquidated in one shot at the first time step, $x=\left(M,0,\cdots,0\right)$. We incur the maximum penalty for this panic move, and it is never optimal. \\In practice, fire sales are generally prohibited in financial markets, as they may trigger circuit breakers and halt trading.
	\item \textbf{Linear sale} (abbreviated \textbf{LIS}): the decision variables all take the same value $x=\left(M/N,\cdots,M/N\right)$. Hence, the quantity to liquidate decreases linearly with time. In financial practice, trades are typically evenly distributed over time.
\end{itemize}

%% file: numalgo2026_UB_intro_v1.tex
When an exact solution cannot be achieved, we seek upper bounds to estimate the range of the optimum. In this section, we transform \PNS\pht into an equivalent \textit{factorable} form, in the sense of McCormick. Using his inequalities, we compute a tight convex relaxation. We aim to develop either an exact solution method based on convexification or tight upper bounds. We begin by introducing a naive upper bound derived from the monotonicity of the penalty function.\vspace{0.2cm}\\
We first define convexity in the context of MINLPs. According to \citet{schm_2021}, a MINLP in a minimization context is called \textit{convex} if its continuous relaxation is a convex problem, i.e., its objective function is a convex function and its admissible set is convex. Hence in a maximization context, we can refer to a MINLP as \textit{convex}, if the continuously relaxed problem has a concave objective function and a convex admissible set. For our class of problems, the (relaxed) admissible set is $\left\{ (x_1,\ldots,x_N)\in [0,M]^N, \phantom{1} S_N = M \right\}$ which is convex, since $S_N$ is linear.

%% file: numalgo2026_UB_mono_v1.tex
The penalty function $g$ is increasing on $\RR_+$ and the variables $x_i$ are non-negative. It directly leads to:\\
\[ f(x)\leq\displaystyle\sum_{i=1}^{N}\Big[p_i- c_i\cdot g(x_i)\Big]x_i \]
We then consider the following optimization problem, denoted \PS\pht:
\begin{equation}
	\big(PS_{N,M}\big)\left\{
	\begin{array}{rl}\label{pbSep}
		\max \; U(x)&=\displaystyle\sum_{i=1}^{N}\Big[p_i- c_i\cdot g\Big(x_i\Big)\Big]\cdot x_i\\
		s.c\;\; \vline & S_N=M\\
		\vline
		&\forall i,\,0\leq x_i\leq M\\
		\vline
		&\forall i,\, x_i\in \NN 	
	\end{array}
	\right.
\end{equation}
This problem is separable and well-defined, and its solution provides an upper bound for \PNS, which we denote by \ub2\pht in the \EN.\\
We now introduce a sufficient condition on $g$ under which $U$ is concave, and consequently, \PS\pht is a convex problem.\vspace{0.2cm}
\begin{lemma}\label{lemXG}
	If the function $x\mapsto x\,g(x)$, defined on $\RR_+$, is strictly  convex, then $U$ is strictly concave and \PS\pht is convex.
\end{lemma} 
\begin{pf}
	If $x\mapsto x\cdot g$ is convex then, for all $i$, since $c_i>0, \;x_i\mapsto p_i\,x_i-c_i\,x_i\,g(x_i)$ is concave as the sum of a linear function $x_i\mapsto p_i\,x_i$ and a concave function of $x_i$.
	Therefore, $U$ is concave as the sum of concave functions of $x_i$.
\end{pf}	
Under the assumption of lemma \ref{lemXG}, \PS\pht is convex. In addition, the constraint qualification is satisfied due to the linearity of the constraint. Therefore, the Lagrange equations provide a global maximum. This leads to the following results:
\begin{theorem}\label{lemCstPrice}
	Under the assumption of Lemma \ref{lemXG}, if the price vector $p$ and the cost coefficients $c$ are constant, the optimal strategy corresponds to the linear sale $\hat{x}=(\frac{M}{N},\cdots,\frac{M}{N})$ and the global maximum is $\mbox{\ub2}=M\,\Big[p-c\,g\big(\frac{M}{N}\big)\Big]$.\\

\end{theorem}
\begin{corollary}
	If the vector $p$ and the cost coefficients $c$ vary, we get an upper bound for \PS. Indeed, let $\overline{p}=\displaystyle\max_i{p_i}$ and $\underline{c}=\displaystyle\min_{i}{c_i}$. Then , for all $i$: $\big[p_i-c_i\,g(x_i)\big]x_i\leq\big[\overline{p}-\underline{c}\,\overline{g}(x_i)\big]x_i$ and theorem \ref{lemCstPrice} applies.\vspace{-0.1cm}\\
\end{corollary}

\begin{pf}
	We define the Lagrangian $L(x,\lambda)$ as: 
	\begin{align*}
		&L(x,\lambda)=\displaystyle\sum_{i=1}^{N}\Big[p_i- c_i\cdot g(x_i)\Big]x_i - \lambda\big(S_N - M\big)\\
		&\nabla L(x,\lambda)=0 \iff \left\{
		\begin{array}{l}
			\forall i, \big[x\,g\big]'(x_i)=\dfrac{p_i-\lambda}{c_i}\\
			S_N-M = 0 \\
		\end{array}
		\right. 
	\end{align*}
	We consider the real function $x\mapsto x\,g(x)$:
	\begin{align*}
		&\mbox{We have: } x\,g(x)=\dfrac{p_i-\lambda}{c_i}\,x+b\\
		&\mbox{Since } g\mbox{ is continuous at 0}\mbox{, then: } b=0\\
		&\mbox{Which yields: }x\,g(x)=\dfrac{p_i-\lambda}{c_i}\,x
	\end{align*}
	There are two cases depending on the price vector $p$ and the cost coefficients $c$.\vspace{0.2cm}\\
	\textbf{Case 1:} the vector $p$ and the coefficients $c$ are constant.\\
	\noindent Lagrange equations yield $x\,g(x)=\dfrac{p-\lambda}{c}\,x$. 
	$g$ is strictly increasing and continuous on $\RR_+$; so it is injective.\\Hence, for a fixed $i$, either $x_i=0$ or $x_i=g^{-1}\big(\frac{p-\lambda}{c}\big)$, which is a constant value independent of $i$.\\In addition, under the assumptions of lemma\ref{lemXG}, $x\cdot g$ is strictly convex, hence $[x\cdot g]'$ is strictly increasing on $\RR_+$ and therefore injective.    
	Let $x=(x_1,\cdots,x_N)$ satisfying Lagrange equations and $i<j$:\vspace{0.1cm}\\
	$\big[x\,\overline{g}\big]'(x_i)=\big[x\,\overline{g}\big]'(x_j)=\dfrac{p-\lambda}{c}$ holds, hence $x_i=x_j$. Since the null vector does not satisfy the constraint, we are left with the linear sale $\forall, i,\; x_i=\frac{M}{N},\lambda=p-c\,g\big(\frac{M}{N}\big)$, and by direct computation of $U$, we obtain \ub2.\vspace{0.2cm}\\
	\textbf{Case 2:} the vector $p$ and the coefficients $c$ vary. We can still get an upper bound for \PS.\\ Indeed, let $\overline{p}=\displaystyle\max_i{p_i}$ and $\underline{c}=\displaystyle\min_{i}{c_i}$. Then, for all $i$: $\big[p_i-c_i\,g(x_i)\big]x_i\leq\big[\overline{p}-\underline{c}\,\overline{g}(x_i)\big]x_i$ \vspace{0.0cm}\\
	Applying the previous result to the right-hand side yields:
	$\mbox{\ub2}\leq\; M\,\Big[\overline{p}-\underline{c}\,g \Big(\frac{M}{N}\Big)\Big]$
\end{pf}

%% file: numalgo2026_UB_convex_v1.tex
We now present a convexification method, based on \citet{mccor_1976}, which provides a general strategy for a wide class of functions called \textit{factorable} functions. We begin with an applicable definition for our mathematical program.
\begin{definition}\label{defFactorable}
	A function $f:\RR^n\mapsto\RR$ is said to be factorable if it can be expressed as the sum of products of univariate continuous functions included in a set $\mathcal{O}$, and whose arguments are either variables, constants, or other factorable functions.  
\end{definition}
An applicable set for our \EN\pht is, for instance, $\mathcal{O}=\{+,x,/,$\textasciicircum{}$,exp, log, sin, cos, asin, acos, atan\}$ (see \citet{mccor_1976} for a more general definition).\\
We seek to establish a convex relaxation of \PNS, in the sense of the definition of section \ref{sectionUB}, using an equivalent factorable form of \PNS. Hence, we factor \PNS\pht (cf. \cite{phdniz_2023} p76-82 for more details) and obtain the maximization problem below. In general, a factorization is not unique.\\
\begin{theorem}
A factorable equivalent reformulation  of \PNS\pht is:
\begin{equation}
	\left\{
	\begin{array}{rcl}\label{pbFactoN}
		\displaystyle\max \;X_{2N}=& \Big(\displaystyle\sum_{i=1}^N p_i\, X_i\Big)& -\Big(c_1\,g(X_1)\,X_1 +c_N\,g(M)\cdot X_N\Big)\\ &&+\displaystyle\sum_{i=2}^{N-1}(-c_i)\cdot\big[X_i\,g(X_{N+i-1})\big]\\
		X_i=&x_i &,\;1\leq i\leq N\\
		X_{N+i-1}=&\displaystyle\sum_{j=1}^i X_j,\;&2\leq i\leq N-1\\
		X_{2N-1}=&\big(\displaystyle\sum_{j=1}^N X_j\big)=&M\\
		&0\leq X_i\leq M &,1\leq i\leq 2N-1\\
		&\forall i,\, X_i\in \NN &	
	\end{array}
	\right.
\end{equation}
\end{theorem}
We consider the continuous relaxation of this reformulation and use the McCormick inequalities to obtain a concave overestimation of its objective function.\\
This reformulation enables us to isolate non-concave terms of the objective function $X_{2N}$, which consists of the univariate function $h:x\mapsto -c_1\,x\cdot g(x)$, and the $(N-2)$ product terms $(X_i,X_j)\mapsto -c_i\,X_i\,g(X_j)$. In addition, the $(N-1)$ constraints are also linear.\\
For a regular (i.e $\mathcal{C}^2$) univariate function $g$, on a bounded segment $[L;U]$, the concave hull of $h$ can be built directly. 
To achieve this, we start by splitting the interval $[L;U]$ into intervals of monotonic convexity for $h$. Using real analysis, we look for inflection points of $h$, where $h''$ admits a root and switches sign, or equivalently where the tangent passes through its graph, which can be formulated: $\exists (a,x)\in[L;U],/ \frac{h(x)-h(a)}{x-a}=h'(x)$,
where $a$ is a given point in the definition interval.\vspace{0.1cm}\\Root finding, whether analytically or by numerical methods, has been studied extensively.\\
We propose here, as in \citet{mccor_1976}, a simple numerical solution relying on Newton's method:\\
We seek roots of the function $H(x)=h(x)-h(a)-(x-a)\,h'(x)$\\
We get: $H'(x)=h'(x)-(x-a)\,h''(x)-h'(x)=-(x-a)\,h''(x)$\\
Newton's method introduces a sequence of $(x_k)_{k\in\NN}$, with $x_0$ and $x_{k+1}=x_k -\frac{H(x_k)}{H'(x_k)}$\\
Thus, we obtain the following recursive relation:  $x_{k+1}=x_k+\frac{1}{h''(x_k)}\Big[\frac{h(x_k)-h(a)}{x_k-a}-h'(x_k)\Big]$\vspace{0.1cm}\\
Consequently, we can split  $[L;U]$ into $[L=l_1;l_2;\cdots;l_i;\cdots l_k=L]$, such that $h$ has monotonic convexity (either strictly convex or concave)  on $[l_i;l_{i+1}]$ for all $i$. Therefore, in a maximization context, we build the concave hull in the following manner:
\begin{itemize}
	\item On concave intervals, $h$ is its own concave hull.
	\item On convex intervals, $[l_{i},l_{i+1}]$, the concave hull corresponds to the segment that connects the points $\big(l_i;h(l_i)\big)$ and $\big(l_{i+1};h(l_{i+1})\big)$. 
\end{itemize}
The reader can refer to \citet{falk_1969} and \citet{horst_1990} p145, for a thorough proof and a geometric interpretation of the convex hull of regular univariate functions $h$.\\
In the remainder of the paper, we will assume, as in Lemma \ref{lemXG} from the previous section, that $h$ is concave. In \EN, we also choose penalty functions $g$ that satisfy this assumption.\vspace{0.2cm}\\
We now tackle the more interesting case of a product of univariate bounded functions, using the McCormick inequalities below:\vspace{0.2cm}
\begin{theorem}\label{thProdMcCo}
	Let U and V be two real univariate functions defined on $[a;b]$, bounded by: $U_a\leq U(x)\leq U_b$ and $V_a\leq V(x)\leq V_b$. Then the following four inequalities apply for all $x$ in $[a;b]$:\\
	$$ U.\,V\geq U\,V_b+U_b\,V-U_b\,V_b$$ 
	$$ U.\,V\geq U\,V_a+U_a\,V-U_a\,V_a$$ 
	$$ U.\,V\leq U_b\,V+U\,V_a-U_b\,V_a$$ 
	$$ U.\,V\leq U_a\,V+U\,V_b-U_a\,V_b$$
	And consequently:
	$$ U.\,V\geq \max\Big(U\,V_b+U_b\,V-U_b\,V_b\,,\,U\,V_a+U_a\,V-U_a\,V_a\Big)$$\vspace{-0.5cm}
	$$ U.\,V\leq \min\Big(U_b\,V+U\,V_a-U_b\,V_a\,,\,U_a\,V+U\,V_b-U_a\,V_b\Big)$$	
\end{theorem}
\begin{pf}
	$\forall x\in [a;b], U_b-U(x)\geq 0$ and $V_b-V(x)\geq 0$. \\Hence $(U_b-U(x))\,(V_b-V(x))\geq 0$. We proceed similarly for the other three inequalities. 
\end{pf}

We combine the previous results in the following corollary to obtain a convex under-estimation (resp. concave over-estimation) of the product $U\circ u(x).\,V\circ v(x)$, where $U,V$ are univariate functions :
\begin{corollary}\label{corolProdMC}
	Let $U,V$ be univariate bounded functions on $[a;b]$. \\Let $e_U$ (resp. $E_U$) and $e_V$ (resp. $E_V$) be their convex (resp. concave) hull function on $[a;b]$: \vspace{-0.2cm}
\[
U(x).\,V(x)\geq \max\Big(e_{U}(x)\,\,V_b+U_b\,e_{V}(x)-U_b\,V_b,\,e_{U}(x)\,V_a+U_a\,e_{V}(x)-U_a\,V_a\Big)
\]\vspace{-0.6cm}\\
and\vspace{-0.1cm}
\[
U(x).\,V(x)\leq \min \Big(U_b\,E_{V}(x)+E_{U}(x)\,V_a-U_b\,V_a\,,\,U_a\,E_{V}(x)+E_{U}(x)\,V_b-U_a\,V_b\Big)
\]
\end{corollary}	
In a maximization setting, we focus on the last inequality of Theorem \ref{thProdMcCo}:
\begin{equation}\label{eqIneqMcCo}
	U.\,V\leq \min\Big(U_b\,V+U\,V_a-U_b\,V_a\,,\,U_a\,V+U\,V_b-U_a\,V_b\Big)
\end{equation}
Let us consider its first term $U_b\,{V}(x)$. 
By definition of the convex (resp.) concave hull, $e_{V}(x)\leq V(x)\leq E_{V}(x)$, for $x\in [a;b]$. If $U_b\geq 0$ then, $U_b\,V(x)\leq U_b\,E_V(x)$, so $U_b\,V$ 
can be replaced by the concave function $U_b\,E_V$ in equation \ref{eqIneqMcCo}.
Conversely, if $U_b< 0$ then $U_b\,V(x)\leq U_b\,e_V(x)$ which is also concave. In that case, $U_b\,V$ can be replaced by  $U_b\,e_V$.\\
By symmetry, we can apply the same reasoning to every component of the inequality \ref{eqIneqMcCo}'s right-hand side.\vspace{0.2cm}\\
Therefore, we can, without loss of generality, adapt theorem \ref{thProdMcCo} according to the sign of bounds $U_a,\;U_b,\;V_a,\;V_b$ to obtain, a concave over-estimation of the product, in a maximization context (and conversely for a minimization).  
Equipped with \citet{mccor_1976} results, we now have the necessary tools to compute a convex relaxation of \PNS\pht in the general case.
Before addressing the general case, we first apply the previous results to the case $N=3$:\\ 
\begin{equation*}
	\big(P_{3,M}\big)\left\{
	\begin{array}{ll}
		\displaystyle\max \;&X_{6}= \Big(\displaystyle\sum_{i=1}^3 p_i\, X_i\Big) -\Big(c_1\,g(X_1)\,X_1 +c_3\,g(M)\cdot X_3\Big) -c_2\cdot X_2\,g(X_{4})\\
		X_{4}=&X_1+X_2,\;2\leq i\leq N-1\\
		X_{5}=&X_1+X_2+X_3=M\\
		0\leq &X_i\leq M \\
		\forall i& X_i\in \NN 	
	\end{array}
	\right.
\end{equation*}
In the general case, the interval $[0;M]$ can be partitioned into sub-intervals of monotonic convexity. On concave sub-intervals, $g$ coincides with its concave hull, whereas on convex sub-intervals $[l_i;l_{i+1}]$, the concave hull of ($g$ is given by the affine function $\alpha_i x + \beta_i$ interpolating the points $A\big(l_i, g(l_i)\big)$ and $B\big(l_{i+1}, g(l_{i+1})\big)$. This uniquely determines the coefficients $\alpha_i$ and $\beta_i$.\\
In the present setting, we restrict attention to a single interval and assume that $g$ is concave on $[0;M]$. We further recall that $g(M)=H$, where the calibration constant is set to $H=0.99$ in \EN. This assumption involves no loss of generality, since the same arguments apply to finer subdivisions of the interval.\\
Here, the non-convex terms are products $-c_2 X_2g(X_4)$ and $-c_1 X_1g(X_1)$.\\
Let $A_2\equiv -c_2\,X_2\, g(X_4), c_2>0$ and $U=X_2$ and $V=-g(X_4)$. We can apply theorem \ref{thProdMcCo} to the product $X_2\,g(X_4)$, with the following bound assumptions:\\ 
\[
\begin{array}{ll}
	U_a=0\leq X_2\leq U_b=M&\\
	V_a=-g(M)=-H\leq -g(X_4)\leq -g(0)=0=V_b&\\\\
\end{array} 
\]
We apply corollary \ref{corolProdMC}, and obtain the following concave overestimation:
\[
A_2\leq c_2\,\min\Big(-U_b\,g(X_4)+X_2\,V_a-U_b\,V_a,-U_a\,g(X_4)+X_2\,V_b-U_a\,V_b\Big)
\] 
It leads to:
\[
\begin{array}{ll}
	A_2\leq  c_2\,\min\Big(-M\,g(X_4)-X_2\,H+M\,H,-0\cdot g(X_4)+X_2\cdot 0-0\cdot 0\Big)&\\
	A_2\leq  c_2\,\min\Big(0,-M\,g(X_4)-H\,X_2+M\,H\Big)&\\\\
\end{array} 
\]
In this case, $V_a=-H<0$, so $V_a\,X_2$ is a linear term.\\
We are left with $-g$, which is convex on $[0;M]$, so its concave hull corresponds to the segment going through $(0,-g(0))$ and $(M,-g(M))$.\\
Hence $-g(X_4)\leq \alpha\cdot X_4+ \beta$, and basic algebra yields
$-g(0)=0=\beta$ and $\alpha=\dfrac{-g(M)+g(0)}{M}=-\dfrac{H}{M}$.\\
We conclude for the $A_2$ term:
\[
\begin{array}{ll}
	A_2\leq  c_2\,\min\Big(0,-H\,X_4-H\,X_2+M\,H\Big)&\\
	A_2\leq  c_2\,H\,\min\Big(0,M-X_2-X_4\Big)&\\\\
\end{array} 
\]
By the same token, we obtain symmetrically for the term $A_1\equiv -c_1\,X_1\,g(X_1)$:
\[
\begin{array}{ll}\label{eqRelaxConvA1}
	A_1\leq  c_1\,H\,\min\Big(0,M-2X_1\Big)&
\end{array} 
\]
Therefore, we derive the following concave overestimation of $f$:
\begin{equation}
	f(x)\leq\displaystyle\sum_{i=1}^{3}p_i\,X_i - c_3\,H\,X_3 + c_1\,H\,\min\Big(0,M-2X_1\Big)+c_2\,H\,\min\Big(0,M-X_2-X_4\Big)
\end{equation}

We assume, as in Section \ref{sectionUB_mono}, under the assumption of Lemma \ref{lemXG} that $x\mapsto x\,g(x)$ is convex. Again, there is no loss of generality, as the $A_1$ term can be convexified as shown above. This leads to the simpler concave overestimation  of $f$:
\begin{equation}\label{eqRelaxConvA2}
	f(x)\leq\displaystyle\sum_{i=1}^{3}p_i\,X_i -c_1\,X_1\,g(X_1) - c_3\,H\,X_3 + c_2\,H\,\min\Big(0,M-X_2-X_4\Big)
\end{equation}
We introduce the following problem:
\begin{equation}
	\big(PC_{3,M}\big)\left\{
	\begin{array}{rcl}
		\displaystyle\max &&X_6 \\
		X_6 &=&\displaystyle\sum_{i=1}^{3}p_i\,X_i -c_1\,X_1\,g(X_1)- c_3\,H\,X_3 +c_2\,H\,\min\Big(0,M-X_2-X_4\Big)\\
		X_4&=&X_1+X_2\\
		X_5&=&X_1+X_2+X_3 = M\\
		\forall i,&\,0\leq &X_i\leq M\\
		\forall i,&\,& X_i\in \NN 	
	\end{array}
	\right.
\end{equation}
The continuous relaxation of $\big(P_{3,M}\big)$ is convex, since $X_6$ is concave. Hence, $\big(P_{3,M}\big)$  is a non-linear convex problem which provides an upper bound for $\big(PNS_{3,M}\big)$. It can be solved directly using a convex solver.\\
However, in practice, this relaxation cannot be expected to be tight, because the affine approximation of $g(X_4)$ on the whole interval $[0;M]$ , with $M$ arbitrary large, is quite poor. Hence, one must restrict the bounding of $g(X_4)$ on a shorter interval to obtain a tighter approximation.
Therefore, in the general case, we need independent bound constraints for each decision variable. This flexibility is required to use \PRC\pht efficiently as the evaluation function of a \BnB\pht algorithm. \vspace{0.2cm}\\
\textbf{General case:} we consider \PNS\pht with independent bound constraints for each decision variable:\\ 
\begin{equation*}
	\left\{
	\begin{array}{rcl}
		\displaystyle\max \;X_{2N}=& \Big(\displaystyle\sum_{i=1}^N p_i\, X_i\Big) -\Big(c_1\,g(X_1)\,X_1 +c_N\,g(M)\cdot X_N\Big) &+\displaystyle\sum_{i=2}^{N-1}(-c_i)\cdot\big[X_i\,g(X_{N+i-1})\big]\\
		X_i=&x_i &,\;1\leq i\leq N\\
		X_{N+i-1}=&\displaystyle\sum_{j=1}^i X_j,\;&2\leq i\leq N-1\\
		X_{2N-1}=&\big(\displaystyle\sum_{j=1}^N X_j\big)=&M\\
		&0\leq L_i\leq X_i\leq U_i\leq M &,1\leq i\leq N\\
		&0\leq a_i\leq X_{N+i-1}\leq b_i\leq M &,2\leq i\leq N-1\\
		&\forall i,\, X_i\in \NN &	
	\end{array}
	\right.
\end{equation*}
With respect to \PNS\pht standard, we added the following bound constraints:
\[
\left\{\begin{array}{ll}
	0\leq L_i\leq X_i\leq U_i\leq M&\\
	0\leq a_i\leq X_{N+i-1}\leq b_i\leq M,\;with\; a_i<b_i&\\
\end{array}
\right. 
\]
We introduce the notation $A_i$ to refer to  the non-concave products of the objective function:
\[
A_i=-c_i\,X_i\,g(X_{N+i-1}),\;\; for \;\; 2\leq i\leq N-1
\]
\begin{rmk} Case $a_i=b_i$ implies $X_{N+i-1}$ is fixed and the corresponding product term $A_i$ becomes linear. We hence assume $a_i<b_i$, for all $i$ in the remainder of the section.
\end{rmk}
We map the bounds of the corollary \ref{corolProdMC} : 
\[
\left\{\begin{array}{ll}
	0\leq L_i\leq X_i\leq U_i\leq M&\vspace{0.1cm}\\
	0\leq a_i\leq X_{N+i-1}\leq b_i\leq M,\;with\; a_i<b_i&\vspace{0.1cm}\\
	-H\leq \underbrace{-g(b_i)}_{V_{a_i}}\leq -g(X_{N+i-1})\leq \underbrace{-g(a_i)}_{V_{b_i}}\leq 0&\vspace{0.1cm}\\
	V_{a_i}=-g(b_i)&\vspace{0.1cm}\\
	V_{b_i}=-g(a_i)&
\end{array}
\right. 
\]
We are in the case where  $V_{a_i}, V_{b_i}$ are negative and apply to the $X_i$ term. As in the case $N=3$, we bound $-g(X_{N+i-1})$ by its concave hull which corresponds to the segment going through $(a_i,-g(a_i))$ and $((b_i,-g(b_i))$.\\
Applying the previous results, we obtain: 
\[
\begin{array}{lll}
	A_i&\leq c_i\,\min\Big(-U_i\,g(X_{N+i-1})+X_i\,V_{a_i}-U_i\,V_{a_i},-L_i\,g(X_{N+i-1})+X_i\,V_{b_i}-L_i\,V_{b_i}\Big)&\\
	&\leq c_i\,\min\Big(U_i\,\big[\alpha_i\,X_{N+i-1}+\beta_i\big]+X_i\,V_{a_i}-U_i\,V_{a_i},L_i\,\big[\alpha_i\,X_{N+i-1}+\beta_i\big]+X_i\,V_{b_i}-L_i\,V_{b_i}\Big)&\\
	&\leq c_i\,\min\Big(U_i\,\alpha_i\,X_{N+i-1}+X_i\,V_{a_i}+U_i\,(\beta_i-V_{a_i}),L_i\,\alpha_i\,X_{N+i-1}+X_i\,V_{b_i}+L_i\,(\beta_i-V_{b_i})\Big)&\\
\end{array} 
\]
where $\alpha_i$ and $\beta_i$ are defined by:
\[
\left\{
\begin{array}{ll}
	\alpha_i=\dfrac{g(a_i)-g(b_i)}{b_i-a_i} <0&\\\\
	\beta_i= \dfrac{a_i\cdot g(b_i)-b_i\cdot g(a_i)}{b_i-a_i}&
\end{array}
\right.
\]

Thus, this leads to the following concave overestimation of the factored problem:
\[
\begin{array}{lll}
	f(x)\leq&\Bigg[\Big(\displaystyle\sum_{i=1}^N p_i\, X_i\Big)&-\Big(c_1\,g(X_1)\,X_1 +c_N\,g(M)\cdot X_N\Big)\\
	&+\displaystyle\sum_{i=2}^{N-1}c_i\cdot\min\Big(&U_i\,\alpha_i\,X_{N+i-1}-g(b_i)\,X_i-U_i\,b_i\,\alpha_i\;,\\
	&&\; L_i\,\alpha_i\,X_{N+i-1}-g(a_i)\,X_i-L_i\,a_i\,\alpha_i\Big)\Bigg]\\\\
\end{array} 
\]  
We conclude with the following problem denoted \PRC:
\begin{equation*}
	(PC_{N,M})\left\{
	\begin{array}{rll}\label{pbConvGnal}
		\displaystyle\max \;X_{2N}=&\Bigg[\Big(\displaystyle\sum_{i=1}^N p_i\, X_i\Big)-\Big(c_1\,g(X_1)\,X_1 +c_N\,g(M)\cdot X_N\Big)&\\
		+\displaystyle\sum_{i=2}^{N-1}c_i\cdot\min\Bigg(&U_i\,\dfrac{g(a_i)-g(b_i)}{b_i-a_i}\,X_{N+i-1}-g(b_i)\,X_i+U_i\,b_i\dfrac{g(b_i)-g(a_i)}{b_i-a_i},&\\
		&L_i\,\dfrac{g(a_i)-g(b_i)}{b_i-a_i}\,X_{N+i-1}-g(a_i)\,X_i+L_i\,a_i\dfrac{g(b_i)-g(a_i)}{b_i-a_i}\Bigg)\Bigg]&\\
		X_i=&x_i ,\;1\leq i\leq N&\\
		X_{N+i-1}=&\displaystyle\sum_{j=1}^i X_j,\;2\leq i\leq N-1 &\\
		X_{2N-1}=&\big(\displaystyle\sum_{j=1}^N X_j\big)=M &\\
		0\leq& L_i\leq X_i\leq U_i\leq M ,1\leq i\leq N&\\
		0\leq& a_i\leq X_i\leq b_i\leq M ,N+1\leq i\leq 2N-2&\\
		\forall i,&\, X_i\in \NN&\\\\	
	\end{array}
	\right.
\end{equation*}
Its (continuously relaxed) objective function is concave, so \PRC\pht is a  convex problem, it is a convex relaxation of \PNS\pht with independent bound constraints. We note that any admissible solution to \PRC\pht is also admissible for \PNS\pht and hence provides a lower bound, based on the original problem's objective function.

%% file: numalgo2026_UB_bnb_v1.tex
The convex relaxation \PRC\pht  described above provides a tight upper bound for \PNS.
We now describe our adapted \BnB\pht algorithm that uses this relaxation to reach the global optimum. We summarize our method as follows: we use \PRC\pht as its evaluation function. Then, we explore the node with the highest mixed-integer convex relaxation and branch on the product term (\textit{i.e.} $-c_i\,X_i\,g(X_{N+i-1})$) that is furthest away from its concave hull. We also apply standard bound reduction methods. The algorithm stops when convergence, within $\epsilon$ precision is reached, or if an exit condition (time, iteration limit) is triggered. It returns either the global optimum or both lower and upper bounds. The algorithm characteristics and pseudo-code are detailed below: 
\begin{itemize}
\item \textbf{Exploration:} we select the node $\mathcal{N}$, which has the highest convex relaxation, among the list of problems $\mathcal{L}$.\\
This approach is driven by convergence. Indeed, $\mathcal{L}$ contains a partition of admissible solutions of \PNS\pht and we try to get closer to the optimum of \PNS, by working on the admissible subset with the highest upper bound of $f$.\\ If we reach, at a given iteration, an upper bound very close (for instance within a given $\epsilon$ distance) to the value of the corresponding solution of \PNS, then every problem in the partition $\mathcal{L}$ would be within $\epsilon$ precision and this admissible solution would hence correspond to a global maximum.We also update the lower bound, if the corresponding convex relaxation leads to better solutions upon evaluation with the original problem's objective function. We refer to $x_{\mathcal{N}}$ as a feasible solution associated with the subproblem at node $\mathcal{N}$ .
\item \textbf{Branching:} Let $\mathcal{N}$ be a feasible node, not necessarily sub-optimal, and let $(UB_\mathcal{N},\;x_{\mathcal{N}})$ be its convex relaxation. We apply the rule of \citet{mccor_1976} and of \citet{falk_1969}, to the current problem. As seen in Section \ref{sectionUB_convex}, non-convexity comes from product terms $U_i\,V_i=-c_i\,X_i\,g(X_{N+i-1})$. For each $i$, we compute the difference between the product and its concave hull, which stems from \PRC, evaluated at node $\mathcal{N}$. We select the index of largest difference. Formally, we seek $i$, such that $\displaystyle\max_{i} E_{U_i\cdot V_i}(x_{\mathcal{N}})-U_i(x_{\mathcal{N}})\,V_i(x_{\mathcal{N}})$, and we branch on auxiliary variable $X_{N+i-1}$.
\item \textbf{Attachment point:} In $\big(PC_{N,M}\big)$, the decision variable $X_{N+i-1}$, at node $\mathcal{N}$ is constrained by:
$a_{i,\mathcal{N}}\leq X_{N+i-1}\leq b_{i,\mathcal{N}}$. Solution $x_{\mathcal{N}}$ is admissible, so we choose $K=X_{N+i-1}(x_{\mathcal{N}})=x_{\mathcal{N},N+i-1}$.\\
Hence, we select the $(N+i-1)^{th}$ coordinate of the optimum as the attachment point.\\We then create two new sub-problems $\mathcal{N}_1$ and $\mathcal{N}_2$ in which constraint  $a_{i,\mathcal{N}}\leq X_{N+i-1}\leq b_{i,\mathcal{N}}$ is replaced by $a_{i,\mathcal{N}}\leq X_{N+i-1}\leq x_{\mathcal{N},N+i-1}$, and $x_{\mathcal{N},N+i-1}+1\leq X_{N+i-1}\leq b_{i,\mathcal{N}}$ respectively.
$\mathcal{N}$ is deleted from $\mathcal{L}$.
$\mathcal{N}_1$ and $\mathcal{N}_2$ are added to $\mathcal{L}$. 
This branching strategy aims at tightening the convex relaxation of the objective function at the optimum $X_{N+i-1}(x_{\mathcal{N}})$, by splitting the variable whose non-convex product is the furthest away from that relaxation. The reader can refer to \citet{falk_1969} for a more thorough justification.     
\item\textbf{Bound-reduction:} This relies on very simple properties of the auxiliary variables.\\Indeed, $\forall i,\, X_i\geq 0$. Hence $X_{N+i-1}=(X_1+\cdots+X_i)\leq K \Rightarrow X_1,X_2,\cdots X_i\leq K$. Conversely, $X_{N+i-1}\geq K \Rightarrow X_{N+i},\cdots X_{2N-2}\geq K$ 
\end{itemize}
\begin{algorithm}[H]
	\caption{Adapted \BnB\pht Algorithm}
	\begin{algorithmic}
		\STATE	\textbf{Initialization:} $\mathcal{L}\leftarrow P_0$; $LB\leftarrow LB_0$; $UB\leftarrow +\infty$; $x_{LB}=x_{LB_0}$
		\WHILE{$\mathcal{L}\neq \varnothing $ and other exit conditions (time etc.)=false)}
		\STATE	\textbf{Exploration:} Solve $PC({\mathcal{N}'})$, for $\mathcal{N}'\in\mathcal{L}$
		\STATE Let $UB(\mathcal{N}')$ be the relaxation value at node $\mathcal{N}'$, i.e the solution of $PC(\mathcal{N}')$  
		\STATE Select the node $\mathcal{N}$, such that $UB(\mathcal{N})=\displaystyle\max_{\mathcal{N}'\in\mathcal{L}} UB(\mathcal{N}') $			
		\STATE	\textbf{Lower bound:} 
		\STATE Let $ LB=\{\max f(x),$ x feasible incumbent of $P_{0}\}$
		\STATE Let $\widehat{x}_{\mathcal{N}'}$ be the optimum of $PC(\mathcal{N}')$ and $f(\widehat{x}_{\mathcal{N}'})$ its evaluation under $P_{0}$ objective function.  
		\STATE if $LB\leq\displaystyle\max_{\mathcal{N}'\in\mathcal{L}} f({\widehat{x}_{\mathcal{N}'}})$, then we update LB.		
		\IF{$(UB(\mathcal{N})-LB)\leq\epsilon$}		
		\STATE Algorithm terminates
		\ELSE
		\STATE \textbf{Branching:} Select $X_i$ such that the difference between the product term $-c_i X_i\,g(X_{N+i-1})$ and its concave hull is maximum.
		\STATE When branching on $X_i$ is not possible, take the second-largest difference, and so on...
		\STATE \textbf{Sub-problems:} 
		Let $P_1$ and $P_2$ be such that : $K= x_{LB}^{\mathcal{N}}(i)$
		\STATE $\Omega_{\mathcal{N}}$: admissible set of $\mathcal{N}$\;  
		\STATE $\Omega_{\mathcal{N}_{1}}=\Omega_{\mathcal{N}}\cap\{x/x_i\leq K\}$ and $\Omega_{\mathcal{N}_{2}}=\Omega_{\mathcal{N}}\cap\{x/x_i\geq K+1\}$
		\STATE \textbf{Bound reduction} on $P_1$ and $P_2$
		\STATE 	$\mathcal{L}\leftarrow\mathcal{L}\cup \mathcal{N}_1\cup \mathcal{N}_2$
		\ENDIF
		\ENDWHILE		
		\STATE \textbf{Output:}
		\IF{$(UB(\mathcal{N})-LB)\leq\epsilon$}
		\STATE $Opt P_0=\{LB,x_{LB}\}$
		\ELSE 
		\STATE Upper bound $UB(\mathcal{N})$
		\STATE $Opt P_0\in[LB;UB(\mathcal{N})]$
		\ENDIF
	\end{algorithmic}
\end{algorithm}

%% file: numalgo2026_EN_intro_v1.tex
In previous sections, we presented several methods to either optimally solve \PNS\pht or to get both lower and upper bounds. We now apply them to a financial application, the large block sale problem and benchmark them against widely used generic solvers. This section is organized as follows: In Section \ref{sectionNEFrame}, we first describe our experimental setup. In line with the LBS framework, we discuss instance sizes, data sources, penalty functions, and justify our choice of solver. Then, we choose to structure \EN\pht by instance size and solution methods (DP-based, local search, or convex relaxation). 
Section \ref{sectionNESmall} is devoted to small and medium-sized instances. We first study DP-based methods in Section \ref{sectionNE_Small_DP}, and then present convexification results in the remainder of Section \ref{sectionNESmall}. 
We apply a similar structure to large and very large instances, for which exact DP is no longer available, in Section \ref{sectionNELarge}. Memory limitations are addressed in Section \ref{sectionNE_Large_Mem}, followed by DP-based methods and convexification techniques in Section \ref{sectionNE_Large_DP}.

%% file: numalgo2026_EN_framework_v1.tex
We detail our framework and begin with machine characteristics and solvers:\\\EN\pht were run on a PC with an Intel Xeon(R) Silver 4114 at 2.20 GHz, two sockets, 20 cores, and 32 GB of RAM. The OS is Linux Ubuntu 18.04 (Bionic Beaver) and the C++ compiler is gcc v9.3.0.\vspace{0.1cm}\\
Next, we list the solvers used as benchmarks and cite corresponding reference articles:\\
\Bar\pht v2022.11.3 (\citet{baron_2022}); \Bon\pht v1.8.8 (\citet{bonmin_2020}); \Cou\pht v0.5.8 (\citet{belo_2009,belo_2013}); \LS\pht v9.5 Linux64 (\citet{beno_2011,beno_2014}) and \Sci\pht v7.0.1 (\citet{scip7_2020}).\vspace{0.2cm}\\
\textbf{Discussion about our choice of solvers:} \Sci\pht is a widely used solver for integer programs. It is flexible, open-source and serves as a useful reference point in the solver ecosystem. \Cou\pht was selected because it implements a solution method closely related to ours, namely a reformulation step and a linear (convex) relaxation embedded in a \BnB\pht algorithm. \Bar\pht is a well-established commercial solver for non-convex global optimization. Lastly, we included \Bon\pht, because it supports a wide range of penalty functions and produced excellent results. However, \Bon\pht is not a global solver for non-convex problems (it returns local maxima only) and does not provide upper bounds.\\Solvers in our pool also have technical limitations.
None of them, except \Bon\pht, supports an arctangent in the penalty function. In particular \Bar\pht does not currently accept trigonometric functions.\\
Furthermore, there are commercial limitations on the number of variables for \Bar, \Cou, and \Bon\pht, although these did not affect the conclusions of the \EN\pht presented in this section.\\
Lastly, in \EN , all solvers received the same \PNS\pht instance to solve, with identical integer settings, time limits, and, when applicable, the same initial solution.\vspace{0.2cm}\\
 
\textbf{Instance size and terminology:} \PNS\pht instances are controlled by the number of decision variables $N$ and the block size $M$. They are presented in the results table either directly as $(N,M)$, or in power form $(N, M)=(10^{a},10^{b}),\; a<b$, with $1\leq a\leq 3$ and $2\leq b\leq 9$. 
We refer to instances as very small to large using the following terminology:\\\textbf{very small}, when $N\leq 10$ (10 or fewer decision variables), \textbf{small}, when $(a,b)\in\big\{(1,2),(1,3),(1,4),(2,3),(2,4)\big\}$, \textbf{medium-sized}, when $(a,b)\in\big\{(1,5),(1,6),(2,5),(3,5)\big\}$, \textbf{large} for $\big\{(1,7),(1,8),(1,9),(2,6),(2,7),(2,8),(2,9),(3,6)\big\}$, and \textbf{very large} for $\big\{(3,7),(3,8),(3,9)\big\}$. We often refer to exponents $(a,b)$ in the \EN.\\
\textbf{Data sources:} to simulate realistic asset price trajectories, we used a Geometric Brownian motion (cf.\citet{hull_2002}) and integrated over a short $\Delta t$, yielding:  $p_{t+\Delta t}=p_t \exp\left[ \Big(\mu -\dfrac{\sigma^2}{2}\Big)\Delta t+\sigma\,\sqrt{\Delta t}\,Z\right] ,\;\hbox{where}\;Z\sim \mathcal{N}(0,1)		
\label{BMeq}$
We start at $p_0=100$, and set varying moments, $\mu\in\{-0.05;0;+0.05\}$ and $\sigma\in\{0.10;0.25;0.70\}$. We also include the constant trajectory $\forall t, p_t=p_0 $.  Thus, for each pair $(N,M)$, we generate 10 deterministic non-negative trajectories, differing in mean and volatility regimes.\\
\noindent\textbf{Results table presentation:} Consistent with the data sources, we divide instances between constant (\textbf{CST}) and varying stock prices (\textbf{AVG}). AVG refers to the average performance metric (i.e. quality or CPU time) over the non-constant trajectories. If the performance varies materially between trajectories, we mention this explicitly.\\
The \textbf{quality} of a lower (respectively, upper) bound is measured by the percentage gap between this bound and the optimum obtained by exact DP. When the optimum is unavailable, the quality is defined as the percentage gap with respect to the best known lower (respectively, upper) bound.\\
In the results tables, we explicitly indicate whether the optimum is attained using the label $Opt$.\\
\textbf{CPU time } is measured in seconds. If an algorithm fails to converge within the prescribed time limit (e.g. 3 hours) or terminates with an error, it is reported as \textit{>3h} (respectively, \textit{error}) in the results table.\\
\textbf{Coefficients $\mathbf{c_i}$:} Financial market economics require a strictly positive asset price, which can be interpreted as the firm’s intrinsic value. In our \EN, we assume $c_i=\beta p_i$, where $\beta=0.9$. 
This implies strictly positive prices and a floor price of 10\% of the current price.\\
As a result, the objective function $f$ is bounded by $f\in[(1-\beta)\,\underline{p}\,M\,;\,\overline{p}\,M]$, where $\underline{p}, (resp.\;\overline{p}$) is the $\displaystyle\min_i {p_i}$,$(resp.\;\displaystyle\max_i {p_i})$.\vspace{0.2cm}\\ 
\textbf{Penalty functions:}  The main techniques of this paper (dynamic programming and convexification) work for any sufficiently regular (e.g. $\mathcal{C}^2$) function. Additional assumptions are only needed for the "monotonic" upper bound in Section \ref{sectionUB_mono}.\\
However, from an economic perspective, the LBS problem requires positive asset prices and an increasing penalty function. We therefore selected, for the \EN, the following real, non-negative, increasing prototype functions, all suitable for our solution methods:
\vspace{0.2cm}\\
$\begin{aligned}[t]
	G\;\left\{
	\begin{array}{l}
		x \mapsto \displaystyle \frac{x}{1+x}\\
		or\; x\mapsto 1-\displaystyle \frac{2}{1+\sqrt{1+x}}\\
		or\; x\mapsto \displaystyle\frac{2}{\pi}\arctan(x)\\			
	\end{array}
	\right.
\end{aligned}$
\vspace{0.2cm}\\
These $G$ functions are  $\mathcal{C}^2$ and bounded on $\mathbb{R}^+$. They satisfy, $G(0)=0$ and $\displaystyle\lim_{+\infty}G=1$. They are concave and the function $ x\mapsto x\cdot G(x)$ is convex, thus fulfilling the assumptions of Section \ref{sectionUB_mono}. 
\\Finally, the penalty function $g$ is defined by $g:x\mapsto G(\eta\, x)$, where $\eta$ is constant and chosen $g(M)=0.99$. A more detailed discussion on calibration is provided in Appendix \ref{sectionCalibration}.
For illustrative purposes, we provide below a graph of one objective function in low dimension : \vspace{-0.5cm}\\ 
\begin{figure}[width=.9\linewidth,cols=1,pos=H]
	\centering	
	\caption{Objective $\big(PNS_{3,100}\big)\;,\; G:x \mapsto\frac{x}{1+x}$}
	\label{figPNS3_100_f2}
	\includegraphics[scale=0.65]{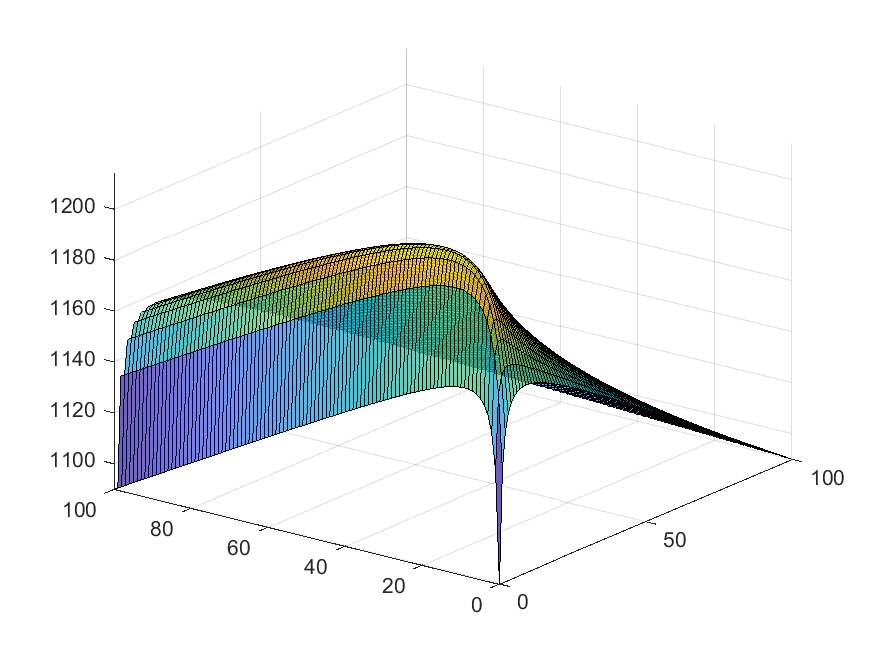}
\end{figure}

%% file: numalgo2026_EN_SmallMed_v2.tex
\vspace{0.2cm}\subsubsection{DP-based methods and other heuristics}\label{sectionNE_Small_DP}
We begin by solving optimally \PNS\pht through exact DP using the exact DP algorithm from Section \ref{sectionDPExact}. The results are similar across all penalty functions, in both the constant (CST) and average (AVG) cases. 
CPU time matches theoretical complexity in $O(N\,M^2)$, ranging from milliseconds up to about $1800\,s$ ($\sim 30\,mn$) for $(a,b) = (2,5)$. 
Hence, while perfectly suited for small instances, exact DP does not scale up. It becomes inefficient when $(a,b)\geq (2,5)$, as solving takes from a few hours ($\sim 5\,h$, when $(a,b) = (1,6)$ or $(3,5)$), up to a few days (e.g $\sim 50\,h$, for $(a,b)=(2,6)$). \vspace{0.2cm}\\ 
Then, we study the TSDP method from Section \ref{sectionLB_TS}, with a first-stage coarse-grained DP with grain $P$, followed by a second-stage search in a funnel of size $\lambda. P$ as the second-stage search neighborhood. We begin by discussing its time complexity.\\ 
\textbf{TSDP Complexity:} $\lambda$ is constant ($\lambda=5$ for \EN) and we note $N=10^a, M=10^b, P=10^c$.\vspace{0.1cm}\\
Coarse-grained DP is $O(10^{a+2b-2c})$ and DP with bounds is $O(10^{2(a+c+1)})$.\vspace{0.1cm}Hence, the total time complexity is $O(10^{\max(a+2b-2c,2a+2c+2)})$.\vspace{0.1cm}\\Its theoretical minimum is reached when $c=\frac{1}{4}(2b-a-2)$, leading to minimal time complexity in $O(10^{\frac{3a}{2}+b+1})$, or $O(N^{\frac{3}{2}}\cdot M)$.
Compared to exact DP (i.e $O(10^{a+2b})$),  this represents a gain of $10^{b-\frac{a}{2}-1}$ , or $\frac{M}{10\sqrt{N}}$.\vspace{0.1cm}\\Furthermore, this gain increases as $M$ grows relative to $N$. Assuming $a$, $b$ and $P$ are integers, we can approximate the theoretical minimum by choosing  $c$ rational, and letting $P=\lfloor 10^c\rfloor$.\\
To select TSDP parameters, we vary $P$ from $10$ to $M/10$, and set $\lambda\in\{1,5\}$.
Empirically, for medium-sized instances, the highest coverage ratio, defined as the number of instances solved to optimality over the total number of instances, was achieved with $P=100$ and a funnel size $\lambda=5$. Using $(\lambda P=500)$, TSDP attains the optimum for more than 95\% of the 117 instances.\vspace{0.2cm}\\ 
With this setting, we compare TSDP performances with that of basic heuristics (\textbf{FS} and \textbf{LIS} from Section \ref{sectionFSLIS}), ILS from section \ref{sectionILS}, the commercial solver \LS\pht (abbreviated \textbf{LS}), and the monotonic upper bound \textbf{\ub2} from Section \ref{sectionUB_mono}. Since results are similar across all penalty functions, we elected to display only the one associated with the prototype $G=\frac{2}{\pi}arctan(x)$, for both CST and AVG samples:\vspace{0.2cm}\\
As seen in tables \ref{tabSmall_TSDP_CST} and \ref{tabSmall_TSDP_AVG}, TSDP reaches the optimum for almost every instance, consistent with prior tables. FS is about $20\%-25\%$ worse than the optimum, and LIS improves as $N$ increases. When $N$ and $M$ are close, it means there are enough decision variables relative to the block size $M$, making a linear sale strategy nearly optimal. This explains why LIS performs better in such cases.\\\LS\pht provides a very good lower bound for every instance, but seldom reaches the optimum. Its coverage ratio is about 10\% for CST and AVG samples.\\
The upper bound \ub2 is not tight, even for small instances. It remains stable in $M$, but deteriorates significantly when $N$ increases, tending towards the natural upper bound $M\,\overline{p}$. For instance in the $(3,5)-\,CST$ instance, the relative difference between $\overline{p}$ and the optimum is about 584\% (almost sevenfold).\\
The ILS local search algorithm performs better on CST than for AVG samples, with the gap widening as instance size increases. In the AVG sample, ILS only marginally outperforms LIS, with results varying significantly by price trajectory $p$. We observed that ILS performance decreases with higher $\sigma$. In high-volatility markets, where most of the sales occur around price spikes, ILS is less efficient and quickly becomes trapped in a local maximum, away from the global optimum. By contrast, the smoother the price trajectory, the better ILS performs. An example for the instance $(a,b)=(3,5)$ is reported in Table \ref{tabSmall_TSDP_traj35} in Appendix \ref{sectionExtraTables}.
Regarding CPU time, results are similar across the penalty functions. TSDP generally converges within a few minutes and solves the instance $(a,b)=(2,6)$ in less than one minute, while the exact DP requires about 50 hours for that instance. \\In addition, TSDP takes slightly longer for CST than AVG. 
The underlying rationale is that, under volatile price trajectories, optimal strategies concentrate trading activity primarily to price spikes, so that many decision variables take small or zero values.
In the LBS context, most trading occurs during price spikes, with minimal trading during dips, which is consistent with market intuition. As a result, the actual number of steps required in the Bellman equation is lower in practice.\\Lastly, ILS converges in under a minute for all instances, while the other heuristics are instantaneous due to closed-form solutions. The corresponding Table \ref{tabSmall_TSDP_CPU} can be found in Appendix \ref{sectionExtraTables}.
\begin{table}[width=.9\linewidth,cols=1,pos=H]
	\centering
	\caption{Quality - small and medium-sized instances - CST}
	\label{tabSmall_TSDP_CST}
	\input{recap_cst_tab4.tex}
\end{table}\vspace{-.5cm}
\begin{table}[width=.9\linewidth,cols=1,pos=H]
	\centering
	\caption{Quality - small and medium-sized instances - AVG}
	\label{tabSmall_TSDP_AVG}	
	\input{recap_avg_tab4.tex}
\end{table}

\subsubsection{Convexification for very small instances}\label{sectionNE_VerySmall_Conv}
We first present results for very small instances ($3\leq N<10$, $10^2\leq M\leq 10^4$, since both the \BnB\pht algorithm and solvers already exhibit specific behaviors, which remain consistent for larger instances. Since only \Bon\pht can handle the arctangent, we work with the other two penalty functions. The time limit is set to 3h ($10\,800s$) for all solvers. As expected, for these instances, both our \BnB\pht algorithm and all solvers, except for \Sci, either reach the optimum or return a very close lower bound. \vspace{0.2cm}\\Upper bounds quality however, proves surprisingly heterogeneous, especially for commercial solvers.\\ While the \BnB\pht algorithm reaches the optimal solution only up to $N\leq 8,M=10^2$, it provides tight upper bounds (within $5\%$ of the optimum for $M=10^2$) and reasonably good results when $M=10^3,\,10^4$ (within $15\%$).\\Both \Bar\phantom{} and \Sci\phantom{} return very poor upper bounds. Their results quickly tend toward the natural upper bound $M\,\overline{p}$ (which stands at a relative distance of about 650-700\%). Thus, they do not provide useful upper bounds within 3 hours of CPU time.\\Interestingly, while \Sci\pht does not perform well as a standalone solver, it performs much better when used within our \BnB\pht algorithm to solve convex relaxations over multiple sub-domains.  In fact, our algorithm outperforms both \Bar\pht and \Cou\pht for the penalty function $x\mapsto\frac{x}{1+x}$. For the other function $x\mapsto 1-\frac{2}{1+\sqrt{1+x}}$, \Cou\pht performs as well as or better than our \BnB. 
CPU time analysis reveals that most instances do not converge within the 3h limit. However, one notable exception is \Bon, which is extremely fast. It converges in just a few seconds for every very small instance. Summary Table \ref{tabVerySmall_UB} is provided in Appendix \ref{sectionExtraTables},\vspace{0.2cm}\\
We can now turn to small and medium-sized instances and present aggregated convexification results.
\subsubsection{Convexification aggregated results} \label{sectionNE_Small_Conv}
To improve clarity, we restrict attention to the two best-performing solvers: \Bon, which achieves the best overall results, and \Cou, which serves as our reference solver. In Table \ref{tabSmall_UB_convex}, which is dedicated to upper bounds, we additionally report the monotonic bound \ub2 introduced in Section \ref{sectionUB_mono} for comparison.\vspace{0.1cm}\\
In the following tables, as \Bon\pht cannot guarantee global optimality, we explicitly indicate when the optimum from exact DP, is attained by the label $Opt$. Results are reported for the CST (constant trajectory) case under both penalty functions. Table \ref{tabSmall_LB_Convex} shows that, as in very small instances, both \Bon\pht and \Cou\pht yield very good admissible solutions.\\
\Bon's coverage ratio is about 61\% (i.e, the optimum is reached 11 times out of 18 instances). It is lower than TSDP, which reaches the optimum for nearly every small and medium-sized instance.\\  
In Table \ref{tabSmall_LB_Convex}, lower bounds returned by our \BnB\pht are less tight than those in the very small case, but still acceptable. For the function $x\mapsto\frac{x}{1+x}$ the gap remains within $12\%$ of the optimum. For the function $x\mapsto 1-\frac{2}{1+\sqrt{1+x}}$ the gap is below 2\%, indicating very tight results.\vspace{0.1cm}\\
However, the picture is less favorable when considering the upper bounds reported in Table \ref{tabSmall_UB_convex}. \\When $N=10$ and the penalty function is $x\mapsto\frac{x}{1+x}$, our \BnB\ yields very tight upper bounds (within a 15\% gap), whereas the bounds obtained by \Cou\pht are several times larger (exceeding a 100\% gap). Conversely, for the other penalty function, the opposite behavior is observed: \Cou\pht outperforms \BnB. The superior performance of \BnB\pht in the former case may be attributed to the highly non-convex shape of $f$ when $G:x\mapsto\frac{x}{1+x}$, as illustrated in Figure \ref{figPNS3_100_f2}.\vspace{0.1cm}\\
When $N\geq 100$, both our algorithm and the reference solver yield uninformative upper bounds, which are close to the natural bound. For instance, the instance $(3,5)$ returns an error when using \Cou.
\begin{table}[width=.9\linewidth,cols=1,pos=H]
	\centering
	\caption{Lower bounds (Convexification) - small and medium-sized instances}
	\label{tabSmall_LB_Convex}
	\input{convbnb_smallmed_lb_f2_tab_4.tex}
\end{table}
\begin{table}[width=.9\linewidth,cols=1,pos=H]
	\centering
	\caption{Upper bounds (Convexification) - small and medium-sized instances}
	\label{tabSmall_UB_convex}
	\input{convbnb_smallmed_ub_f26_tab_1.tex}
\end{table}
Computation time analysis reveals an even more interesting pattern. As indicated in Tables \ref{tabSmall_UB_Convex_CPU_f2} and \ref{tabSmall_UB_Convex_CPU_f6}, no instance, whether solved by our \BnB\pht or by \Cou\pht, converges within 3h, which is consistent with previous observations.\\\Bon's empirical complexity appears polynomial in $N$, but decreasing in $M$, which is counter-intuitive. However, \Bon\pht does not converge in 3h when $N=10^3$.
The comparison between \Bon\pht and exact DP, (with complexity in $O(NM^2)$) proves also very insightful. Both methods solve instantaneously and exactly the instance $(1,2)$.
As $N$ increases (with fixed $M$), DP becomes faster than \Bon, which seems to be at least in $N^2$.\\On the other hand, as $M$ increases (with fixed $N$), then \Bon\pht converges much faster than DP, but not necessarily towards the optimum. The most significant differences are highlighted in bold in the CPU tables. We emphasize two specific instances that illustrate this pattern:\\
$\bullet$ Instance (1,6): \Bon\pht converges instantaneously,while DP takes about 6h. Although \Bon\pht does not reach the optimum, it yields an excellent lower bound, with a difference below $0.01\%$.\\
$\bullet$ Instance (2,3): solved optimally and instantly by DP, but takes around 1 hour and 15 minutes with \Bon.
\\These differences highlight the heterogeneous performance of both our \BnB\pht algorithm and the commercial solvers. They underscore the practical importance of estimating complexity (even empirically) before selecting a solution method.
\begin{table}[width=.9\linewidth,cols=1,pos=H]
	\centering
	\caption{CPU time (Convex.) - Small and medium-sized instances - $x\mapsto\frac{x}{1+x}$}
	\label{tabSmall_UB_Convex_CPU_f2}
	\input{convbnb_smallmed_cpu_f2_tab_2.tex}	
\end{table}
\vspace{-0.3cm}
\begin{table}[width=.9\linewidth,cols=1,pos=H]
	\centering
	\caption{CPU time (Convex.) - Small and medium-sized instances - $x\mapsto 1-\frac{2}{1+\sqrt{1+x}}$}
	\label{tabSmall_UB_Convex_CPU_f6}
	\input{convbnb_smallmed_cpu_f6_tab_2.tex}
\end{table}
In the next section, we present the results of the \EN\pht for large and very large instances.

%% file: recap_cst_tab4.tex
\begin{tabular}{|c|c|c|c|c|c|c|c|}
\hline
\multicolumn{2}{|c}{LB Quality}&\multicolumn{5}{c|}{CST, $\frac{2}{\pi}arctan(x)$}&{UB Quality}\\
\hline
N&M&FS&LIS&ILS&LS&TSDP&\ub2 \\\hline
$10^1$&$10^2$&20.42&7.81&0.41&\textit{Opt}&\textit{Opt}&38.19 \\\hline
$10^1$&$10^3$&20.52&7.92&0.02&<0.01&\textit{Opt}&38.02     \\\hline
$10^1$&$10^4$&20.52&7.92&<0.01&<0.01&\textit{Opt}&38.02            \\\hline
$10^1$&$10^5$&20.52&7.92&<0.01&<0.01&\textit{Opt}&38.02            \\\hline
$10^1$&$10^6$&20.52&7.92&<0.01&<0.01&\textit{Opt}&38.02            \\\hline
$10^2$&$10^3$&25&2.12&1.04&<0.01&\textit{Opt}&364.61       \\\hline
$10^2$&$10^4$&25&2.14&0.13&0.01&\textit{Opt}&364.56              \\\hline
$10^2$&$10^5$&25.01&2.14&<0.01&0.01&\textit{Opt}&364.56    \\\hline
$10^3$&$10^5$&25.49&0.23&0.12&0.10&\textit{Opt}&558.72            \\\hline
\end{tabular}

%% file: recap_avg_tab4.tex
\begin{tabular}{|c|c|c|c|c|c|c|c|}
\hline
\multicolumn{2}{|c}{LB Quality}&\multicolumn{5}{c|}{AVG, $\frac{2}{\pi}arctan(x)$}&{UB Quality}\\
\hline
N&M&FS&LIS&ILS&LS&TSDP&\ub2 \\\hline
$10^1$&$10^2$&21.66&10.47&0.96&\textit{Opt}&\textit{Opt}&139.4              \\\hline
$10^1$&$10^3$&21.77&10.58&0.72&<0.01&\textit{Opt}&139.07   \\\hline
$10^1$&$10^4$&21.77&10.59&0.71&<0.01&\textit{Opt}&139.07   \\\hline
$10^1$&$10^5$&21.77&10.59&0.71&<0.01&\textit{Opt}&139.07   \\\hline
$10^1$&$10^6$&21.77&10.59&0.71&<0.01&<0.01&139.07   \\\hline
$10^2$&$10^3$&28.48&7.77&5.26&<0.01&\textit{Opt}&419.26    \\\hline
$10^2$&$10^4$&28.49&7.8&4.04&0.01&\textit{Opt}&419.13            \\\hline
$10^2$&$10^5$&28.50&7.8&3.97&0.01&<0.01&419.13             \\\hline
$10^3$&$10^5$&30.49&5.98&7.39&0.14&<0.01&576.05           \\\hline
\end{tabular}

%% file: convbnb_smallmed_lb_f2_tab_4.tex
\begin{tabular}{|c|c|c|c|c||c|c|c|}
	\hline
	\multicolumn{2}{|c|}{LB Quality}&\multicolumn{3}{c||}{CST, $\;\;G:x\mapsto\frac{x}{1+x}$}&\multicolumn{3}{c|}{CST, $\;\;G:x\mapsto 1-\frac{2}{1+\sqrt{1+x}}$}\\
	\hline
	N&M&\textit{B}\&\textit{B}&\Bon&\Cou&\textit{B}\&\textit{B}&\Bon&\Cou     \\\hline
	$10^1$&$10^2$&0.03&$Opt$&0.03&0.04&$Opt$&<0.01      				\\\hline
	$10^1$&$10^3$&0.19&$Opt$&<0.01&0.19&$Opt$&<0.01             \\\hline
	$10^1$&$10^4$&0.28&$Opt$&<0.01&0.22&$Opt$&<0.01             \\\hline
	$10^1$&$10^5$&0.58&$Opt$&<0.01&0.3&$Opt$&<0.01              \\\hline
	$10^1$&$10^6$&0.59&<0.01&<0.01&0.36&<0.01&<0.01             \\\hline
	$10^2$&$10^3$&7.21&$Opt$&0.06&1.54&$Opt$&0.03             				\\\hline
	$10^2$&$10^4$&6.62& <0.01&<0.01&1.6&<0.01&<0.01              \\\hline
	$10^2$&$10^5$&6.52&$Opt$&<0.01&1.13&<0.01&<0.01             \\\hline
	$10^3$&$10^5$&11.84&<0.01&<0.01&1.3&<0.01&<0.01             \\\hline
\end{tabular}

%% file: convbnb_smallmed_ub_f26_tab_1.tex
\begin{tabular}{|c|c|c|c|c||c|c|c|}
	\hline
	\multicolumn{2}{|c|}{UB Quality}&\multicolumn{3}{c||}{CST, $\;\;G:x\mapsto\frac{x}{1+x}$}&\multicolumn{3}{c|}{CST, $\;\;G:x\mapsto 1-\frac{2}{1+\sqrt{1+x}}$}\\
	\hline
	N&M&\textit{B}\&\textit{B}&\Cou&\ub2&\textit{B}\&\textit{B}&\Cou&\ub2     \\\hline
	$10^1$&$10^2$&1.23&168.65&36.89&13.88&0.03&11.13      				\\\hline
	$10^1$&$10^3$&10.55&207.56&36.71&139.99&24.04&11.09             \\\hline
	$10^1$&$10^4$&13.58&211.86&36.71&178.72&48.3&11.09             \\\hline
	$10^1$&$10^5$&15.1&219.01&36.71&195.14&110.25&11.09              \\\hline
	$10^1$&$10^6$&14.78&215.39&36.71&254.51&110.86&11.09             \\\hline
	$10^2$&$10^3$&607.33&609.83&292.02&716.12&753.22&58.72             				\\\hline
	$10^2$&$10^4$&609.52&609.64&291.91&743.93&746.61&58.58              \\\hline
	$10^2$&$10^5$&609.56&607.22&291.91&722.08&741.91&58.57             \\\hline
	$10^3$&$10^5$&605.34&error&548.19&750.64&error&192.78             \\\hline
\end{tabular}

%% file: convbnb_smallmed_cpu_f2_tab_2.tex
\begin{tabular}{|c|c|c|c|c|c|}
	\hline
	\multicolumn{2}{|c|}{\textbf{CPU}}&\multicolumn{4}{c|}{CST, $\;\;G:x\mapsto\frac{x}{1+x}$}\\
	\hline  	
	N&M&\BnB&\Bon&\Cou&Exact DP	   		\\\hline  	
	$10^1$&$10^2$&$>3h$&0.37&$>3h$&<0.01 	\\\hline
	$10^1$&$10^3$&$>3h$&0.38&$>3h$&0.03			\\\hline
	$10^1$&$10^4$&$>3h$&0.74&$>3h$&1.77			\\\hline
	$10^1$&$10^5$&$>3h$&0.58&$>3h$&181			\\\hline
	$10^1$&$10^6$&$>3h$&0.56&$>3h$&18 141		\\\hline
	$10^2$&$10^3$&$>3h$&\textbf{\textit{Opt} 500}&$>3h$&\textbf{0.19}			\\\hline
	$10^2$&$10^4$&$>3h$&423&$>3h$&18.83			\\\hline
	$10^2$&$10^5$&$>3h$&\textbf{\textit{Opt} 324} &$>3h$&\textbf{1829}			\\\hline
	$10^3$&$10^5$&$>3h$&$>3h$&$>3h$&17 550		\\\hline
\end{tabular}

%% file: convbnb_smallmed_cpu_f6_tab_2.tex
\begin{tabular}{|c|c|c|c|c|c|}
	\hline
	\multicolumn{2}{|c|}{\textbf{CPU}}&\multicolumn{4}{c|}{CST, $\;\;G:x\mapsto 1-\frac{2}{1+\sqrt{1+x}}$}\\
	\hline
	N&M&\BnB&\Bon&\Cou&Exact DP	   		\\\hline  	
	$10^1$&$10^2$&6740&0.39&$>3h$& <0.01			\\\hline
	$10^1$&$10^3$&$>3h$&0.54&$>3h$& 0.02				\\\hline
	$10^1$&$10^4$&$>3h$&0.53&$>3h$&1.77         		\\\hline
	$10^1$&$10^5$&$>3h$&0.85&$>3h$&176         			\\\hline
	$10^1$&$10^6$&$>3h$&\textbf{0.26}&$>3h$&\textbf{\small 17 876}         		\\\hline
	$10^2$&$10^3$&$>3h$&\textbf{\textit{Opt} 4516}&$>3h$ &\textbf{0.18}         		\\\hline
	$10^2$&$10^4$&$>3h$&3146&$>3h$&17.55          		\\\hline
	$10^2$&$10^5$&$>3h$&2649&$>3h$&1752          		\\\hline
	$10^3$&$10^5$&$>3h$&$>3h$&$>3h$&17 949       		\\\hline
\end{tabular}

%% file: numalgo2026_EN_Large_v1.tex
At this scale, exact DP is no longer applicable, as it is too costly in both time and memory. Furthermore, several methods discussed in previous sections do not converge for the same reasons. We first address memory limits in relation to instance size.

\subsubsection{Memory limits}\label{sectionNE_Large_Mem}
We determine the maximum tractable size for DP-based methods, whose space complexity is $O(N\,M)$. A double occupies 8 bytes of heap memory and 1GB=$1024^3=2^{30}$ bytes. Given our $32\,GB$ RAM machine, we get:\\$8\,N\,M=\;8\cdot 10^{a+b}\leq 32\cdot2^{30} \implies \;a+b\leq 32\cdot\dfrac{\ln(2)}{\ln(10)}\approx 9.63<10$. 
Thus, our memory constraint corresponds to a maximum total exponent of 9. We cannot go beyond $(a,b)=(1,8),\;(2,7),\;(3,6)$ for exact DP or TSDP.\\
In the exact DP algorithm, the $N\times M$ matrix of optimal values is fully allocated. However, in DP with bounds (i.e. the  TSDP second stage) the space complexity is $O(N^2\,P)$. Consequently, the standard $N\times M$ matrix becomes memory-intensive when $N\,P\gg M$.
\\Since the TSDP optimum matrix is only filled within the funnel bounds, one can avoid storing zero entries by concatenating the data into a linear vector, as illustrated in Figure \ref{fig:StructDP}. The space saved can be shown to be approximately $\frac{M}{2\lambda\,P\,N}$.\\
In \EN\pht on (very) large instances, we typically have $M\gg N$, by at least a factor of $10^3$, and up to  $10^6$. Therefore, this adapted data structure would likely reduce memory usage by at least one order of magnitude.
We discuss this further as a research perspective in Section \ref{sectionCCL}.
\begin{figure}[width=.9\linewidth,cols=1,pos=H]
	\caption{Comparison of DP matrix structures}
	\label{fig:StructDP}
	\begin{center}
		\begin{tikzpicture}[scale=.9]	
			\tikzstyle{etape}=[text width=2cm,align=left]
			\tikzstyle{etape2}=[text width=5cm, align=left]
			
			\node[etape] at(-7,5){$n\backslash m$};
			\node[etape] at(-7,4.25){$1$};
			\node[etape] at(-7,3.5){$2$};
			\node[etape] at(-7,2.75){$\vdots$};
			\node[etape] at(-7,2){$N$};
			\node[etape] at(-5.75,5){$1$};
			\node[etape] at(-5.0,5){$L_1$};
			\node[etape] at(-3.75,5){$L_2$};
			\node[etape] at(-2.25,5){$U_1$};
			\node[etape] at(-1.2,5){$U_2\;\cdots U_N$};
			\draw (-7,5.25) -- (-7,1.5);
			\draw (-8,4.7) -- (0,4.7);
			\draw (-7,4.7) rectangle (0,1.5);
			
			\node[etape2] at(-5.5,5.75){\textbf{Current structure :}};
			\draw[fill=green!45] (-6.0,4.5) rectangle (-3.0,4.0);
			\node[etape] at(-3.8,4.2){$O_{1,m}$};
			\draw[fill=red!45] (-4.75,3.75) rectangle (-2.0,3.25);
			\node[etape] at(-2.8,3.45){$O_{2,m}$};
			\node[etape] at(-1.25,2.75){$\vdots$};
			\draw[fill=blue!45] (-2.75,2.25) rectangle (-1.0,1.75);
			\node[etape] at(-1.35,1.95){$O_{N,m}$};
			\node[etape] at(-5.5,2.55){$\Big(0\Big)$};
			\node[etape] at(-0.3,4.0){$\Big(0\Big)$};
			
			\node[etape2] at(-5.5,0.25){\textbf{Target structure:}};
			\draw[fill=green!45] (-8.0,-0.5) rectangle (-5.0,-1.0);
			\draw[fill=red!45] (-4.95,-0.5) rectangle (-2.2,-1.0);
			\node[etape] at(-0.7,-0.8){$\cdots\cdots$};			
			\draw[fill=blue!45] (-0.15,-0.5) rectangle (+1.6,-1.0);
			\node[etape] at(-5.75,-0.8){$O_{1,m}$};
			\node[etape] at(-2.85,-0.8){$O_{2,m}$};
			\node[etape] at(+1.25,-0.8){$O_{N,m}$};
			\draw (-8.5,-0.25) rectangle (2,-1.25);	
		\end{tikzpicture}
	\end{center}
\end{figure}
We now present aggregated results of DP-based methods for large and very large instances.

\subsubsection{DP-based methods and convexification}\label{sectionNE_Large_DP}
We adopt the same conventions as in Section \ref{sectionNESmall}. Quality is defined as the relative gap with respect to the best known lower bound, since the optimum is unavailable. The corresponding method is indicated in the third column, and, for the TSDP method, the funnel size $\lambda P$ is reported in parentheses. Finally, the label \textit{memory} indicates that a solution method could not be executed due to memory limitations.\\
We present the results for the remaining penalty function  $x\mapsto\frac{2}{\pi}\arctan(x)$ for both CST and AVG cases.
The solution method TSDP provides the best lower bound when available. The difference with \LS\pht is very small (less than $1\%$), although the gap tends to increase with $N$.\\When TSDP is not available due to memory constraints, ILS provides the best lower bound in the CST case, slightly outperforming \LS. In the AVG case, however, \LS\pht provides better results by about 1-7\%. In both cases, the gap increases with $N$.\\As previously noted in Section \ref{sectionNESmall}, ILS results depend on the price trajectory. They are excellent in the CST case and deteriorate with higher $\sigma$, more so with high drift $\mu$. In financial terms, this heuristic is less suitable for highly volatile markets, especially when prices trend upward.
Lastly, the firesale (FS) strategy underperforms the best lower bound by about 20-30\%, and the Linear sale by 1 to 10\%, across the board. The monotonic upper bound \ub2 is  only relevant in the CST case for $N=10$, and nowhere else.
The corresponding Tables \ref{tabLarge_CST_f5} and \ref{tabLarge_AVG_f5} are provided in Appendix \ref{sectionExtraTables}. 
\textbf{CPU time:} TSDP has time complexity $O(N^{\frac{3}{2}}\cdot M)$ and converges within 1 to 15 minutes (worst case is $32\;min$ when $(a,b)=(2,7)$). Therefore, it is practically usable for large instances. However, due to memory constraints discussed in Section \ref{sectionNE_Large_Mem}, it is no longer available when $a+b\geq 10$.\\ The ILS method is relatively fast, but less accurate for large instances.\\ Trajectory analysis reveals that ILS converges faster for highly fluctuating price paths, as it quickly becomes trapped in the first local maximum it encounters. TSDP also shows a decreasing computation time as $\sigma$ increases.\\ Lastly, FS, LIS and \ub2 are computed in quasi-constant time.
We finally discuss the convexification results for large instances.\\ 
As mentioned in Section \ref{sectionNEFrame}, only \Bon\pht supports the arctangent penalty. In Section \ref{sectionNE_Small_Conv}, we saw that for $N>10^2$, neither our \BnB\pht nor the solver \Cou\pht provided meaningful upper bounds.\\
For large and very large instances, we additionally face memory issues, and only a handful of instances converge. We compare the lower bounds obtained by TSDP, ILS, and \Bon.\\\Bon\pht converges for instances ranging from $(2,6)$ to $(2,9)$, while it fails on the remaining instances. Only two instances, namely (2,6) and (2,7), are solved by both TSDP and \Bon. In both cases, TSDP yields slightly better lower bounds. For the instances $(2,8)$ and $(2,9)$, \Bon\pht provides the strongest lower bound, again with a marginal improvement over ILS. For instances ranging from $(3,7)$ through (3,9), ILS is the only method that returns a lower bound. Finally, the monotonic bound \ub2 yields an informative upper bound only when $N=10$.
Consequently, in the absence of a proof of optimality or sufficiently tight upper bounds, the optimal values for the large and very large instances cannot be estimated accurately. The corresponding numerical results are reported in Table \ref{tabLarge_Convex_f5} in Appendix \ref{sectionExtraTables}, 
where the best lower bound among TSDP, ILS, and \Bon\pht is indicated by a cross ($\times$).

%% file: numalgo2026_ccl_v3.tex
We solved the non-convex integer problem \PNS\pht, either exactly or approximately, while pursuing two directions: the first path consisted of recursively decomposing the problem into simpler subproblems with fewer decision variables. We proved that the Bellman equation applies and established an exact solution algorithm using DP, in pseudo-polynomial time.
However, it is efficient in practice only for small and medium-sized instances.\\
We then introduced the TSDP method, with a better time complexity than exact DP. 
We also proposed a local search algorithm, denoted ILS, which is very accurate in the constant price case.\\\EN\pht revealed that TSDP is strong for many tested instances when memory permits.
 \\Lastly, under additional assumptions on the penalty function, we derived an analytical upper bound for \PNS, referred to as \ub2. However, the \EN\pht showed that this bound is not tight.\vspace{0.2cm}\\
The second direction we explored was convexification. Using McCormick inequalities, we presented a tight convex relaxation of \PNS, and developed an adapted \BnB\pht to solve the original non-convex problem. In the \EN, however, only very small instances converged, both for our \BnB\pht and for most benchmark solvers, within the three-hour time limit. Our adapted \BnB\pht sometimes gives better upper bounds than benchmark solvers, but often fails to provide a good upper bound.
\\We complete this study by highlighting two promising directions for future research: First, as mentioned in Section \ref{sectionNELarge}, significant memory savings can be achieved by designing an adapted data structure for DP-based algorithms. Second, our \BnB\pht can be further improved through tighter convex relaxations and more advanced branching techniques. 

%% file: numalgo2025_appendix_bellman.tex
\begin{pf} By induction on n.\\
	Let n be fixed and $(n<m)$. For $n=1$, we have to sell the whole block at the only allowed time n. 
	\begin{equation*}
			\begin{split}
					O_{1,m}&=f(m) \\
					& = \big[p_1-c_1\,g\big(m\big)\big]\,m \\
					& = \displaystyle\max_{i \in [\![0;m]\!]}\left\{ \big[p_1-c_1\, g(m)\big]i\right\},\;since\;(p_1-c_1\,g(m))>0\\
					& = \displaystyle\max_{i \in [\![0;m]\!]}\left\{O_{0,m-i}  + \big[p_1-c_1\, g(m)\big]i\right\},\;since\;O_{0,m-i}=0\; \forall i\\
				\end{split}
		\end{equation*}
	which proves the result for n=1. Let us suppose the result be true for a given n.\\
	Let $g_i(x)\equiv\Big[p_i-c_i\,g\Big(\displaystyle\sum_{k=1}^{i} x_k\Big)\Big]$
	and let $(\widehat{x_1},\cdots,\widehat{x}_{t+1})$ be an optimal solution of $(P_{t+1,n})$.\\By definition: 
	$O_{n+1,m}=\displaystyle\sum_{i=1}^{n}g_i(\widehat{x})\,\widehat{x_i} + \big[p_{n+1}-c_{n+1}\,g^(m)\big]\,\widehat{x}_{n+1} $\\
	Since $\displaystyle\sum_{i=1}^{n}\widehat{x_i}=m -\widehat{x}_{n+1}$, the first $n$ coordinates of $\widehat{x}$ constitute an admissible solution of $(P_{n,m-\widehat{x}_{n+1}})$. By our induction assumption,  its value function is therefore bounded by $O_{n,m-\widehat{x}_{n+1}}$. It yields\\
	\begin{equation*}
			\begin{split}
					O_{n+1,m}&\leq O_{n,m-\widehat{x}_{n+1}} + \big[p_{n+1}-c_{n+1}\,g(m)\big]\,\widehat{x}_{n+1}\\
					&\leq\displaystyle\max_{i \in [\![0;m]\!]} \left\{ O_{n,m-i}  + \big[p_{n+1}-c_{n+1}\,g(m)\big]i\;\right\}
				\end{split}
		\end{equation*}
	Conversely, for all $0\leq i\leq m$, $(P_{n,m-i})$ admits a solution. Hence,\\
	\[
	\exists(x_1,\cdots,x_n)\;s.t \left\{
	\begin{array}{l}
			O_{n,m-i}=\displaystyle\sum_{j=1}^{n}g_j(x)\,x_j \\
			\displaystyle\sum_{j=1}^{n}x_j=m-i	
		\end{array}
	\right.
	\]
	Appending $x$ by $i$ $(x_1,\cdots,x_n,\mathbf{i})$ leads to an admissible for $(P_{n+1,m})$. Hence, its value function is bounded by $O_{n+1,m}:$\\
	\begin{align*}
			f(x) &\leq O_{n+1,m}\\
			\displaystyle\sum_{j=1}^{n}g_j(x)\,x_j + \big[p_{n+1}-c_{n+1}\,g(m)\big]\,i & \leq O_{n+1,m}\\
			O_{n,m-i} + \big[p_{n+1}-c_{n+1}\,g(m)\big]\,i & \leq O_{n+1,m}
		\end{align*}
	Since this inequality can be established for any $i$ in $[\![0;m]\!]$, we take the maximum on $i$ and conclude:\\
	\[
	\displaystyle\max_{i \in [\![0;m]\!]} \left\{O_{n,m-i} + \big[p_{n+1}-c_{n+1}\,g(m)\big]\,i\right\} \leq O_{n+1,m}
	\]
	which proves the results for $(P_{n+1,m})$.
\end{pf}

%% file: numalgo2026_appendix_np.tex
\noindent The complexity status of \PNS\pht remains open. We only establish that its decision version belongs to NP and provide a pseudo-polynomial dynamic programming algorithm of complexity $O(N\,M^2)$.\vspace{0.2cm}\\
\textbf{Proposition:} The decision version of \PNS\pht belongs to NP.\\
\begin{pf}
A certificate consists of vector $x=(x_1,\cdots,x_N)$ satisfying $x_i\in\NN$.\\
Given such a certificate, feasibility can be verified by checking that $\sum_{i=1}^N x_i = M$ and by computing the cumulative sums $S_i=\sum_{k=1}^i x_k$, which requires $O(N)$ arithmetic operations.The objective value 
\[
	\sum_{i=1}^N \big( p_i - c_i \, g\!\left( S_i \right) \big)x_i
\]
can then be evaluated in polynomial time, assuming that g is polynomial-time computable. Finally, one checks whether the objective value is at least a prescribed threshold K. Hence the decision version of \PNS\pht can be verified in polynomial time and therefore belongs to NP. 
\end{pf}

%% file: numalgo2025_appendix_extratables.tex
\begin{table}[width=.9\linewidth,cols=1,pos=H]
	\centering
	\caption{CPU Time -  Exact DP}
	\label{tabCpuCompLB}
	\input{recap_cpu_lb_tab_1.tex}
\end{table}
\begin{table}[width=.9\linewidth,cols=1,pos=H]
	\centering
	\caption{CPU time - TSDP}
	\label{tabSmall_TSDP_CPU}
	\input{ts1_cpu_tab2.tex}
\end{table}
\begin{table}[width=.9\linewidth,cols=1,pos=H]
	\centering
	\caption{Instance (3,5) - Quality by $p$-trajectory}
	\label{tabSmall_TSDP_traj35}	
	\input{focus_ILS_tab2.tex}
\end{table}
\begin{table}[width=.9\linewidth,cols=1,pos=H]
	\centering
	\caption{Upper Bounds (Convexification) - Very small instances- $x\mapsto\frac{x}{1+x}$}
	\label{tabVerySmall_UB}
	\input{convbnb_verysmall_ub_f2_tab_4.tex}
\end{table}
\begin{table}[width=.9\linewidth,cols=1,pos=H]
	\centering
	\caption{Quality - Large instances - CST}
	\label{tabLarge_CST_f5}
	\input{recap_cst_largeinstances_tab4.tex}
\end{table}
\vspace{-0.2cm}
\begin{table}[width=.9\linewidth,cols=1,pos=H]
	\centering
	\caption{Quality - Large instances - AVG}
	\label{tabLarge_AVG_f5}
	\input{recap_avg_largeInstances_tab4.tex}
\end{table}
\begin{table}[width=.9\linewidth,cols=1,pos=H]
	\centering
	\caption{Quality (Convex.) - Large instances}
	\label{tabLarge_Convex_f5}
	\input{recap_conv_large_f5_tab_5.tex}
\end{table}

%% file: recap_cpu_lb_tab_1.tex
\begin{tabular}{|c|c|c|c|c||c|c|c|}
	\hline
	\multicolumn{2}{|c|}{CPU}&\multicolumn{3}{c||}{CST,$\;G: x\mapsto\frac{2}{\pi}arctan(x)$}&\multicolumn{3}{c|}{AVG,$\;G: x\mapsto\frac{2}{\pi}arctan(x)$}\\
	\hline
	N&M&DP&TSDP&ILS&DP&TSDP&ILS     \\\hline
	$10^1$&$10^2$&$<\,0.01$&$<\,0.01$&$<\,0.01$&$<\,0.01$&$<\,0.01$&$<\,0.01$ 		\\\hline
	$10^1$&$10^3$&0.03&0.02&$<\,0.01$&0.02&0.02&$<\,0.01$            \\\hline
	$10^1$&$10^4$&1.78&0.12&$<\,0.01$&1.78&0.11&$<\,0.01$             \\\hline
	$10^1$&$10^5$&179&0.21&$<\,0.01$&178&0.19&$<\,0.01$              \\\hline
	$10^2$&$10^3$&0.18&0.15&0.01&0.18&0.15&$<\,0.01$			\\\hline
	$10^2$&$10^4$&17.68&2.12&0.22&17.63&1.8&0.07			\\\hline
	$10^2$&$10^5$&1786&12.90&0.73&1763&6.77&0.14              \\\hline
	$10^3$&$10^5$&17576&213.36&35.42&17856&138.24&2.92            \\\hline
	$10^1$&$10^6$&18261&2.10&$<\,0.01$&18142&2.02&$<\,0.01$ \\\hline
	$10^2$&$10^6$&$\sim 180\,000$&39.58&1.31&$\sim 180\,000$&28.1&0.21             \\\hline

\end{tabular}

%% file: ts1_cpu_tab2.tex
\begin{tabular}{|c|c|c|c|c|c|c|c|}
\hline
\multicolumn{2}{|c|}{CPU}&\multicolumn{2}{c|}{$\frac{x}{1+x}$}& \multicolumn{2}{c|}{$\frac{2}{\pi}\arctan(x)$}& \multicolumn{2}{c|}{$1-\frac{2}{1+\sqrt{1+x}}$}\\[2pt]
\hline
N&M&CST&AVG&CST&AVG&CST&AVG \\
\hline\hline
$10^1$&$10^2$&<0.01&<0.01&<0.01&<0.01&<0.01&<0.01             \\\hline
$10^1$&$10^3$&0.02&0.017&0.017&0.016&0.016&0.015        \\\hline
$10^1$&$10^4$&0.127&0.108&0.118&0.102&0.143&0.076       \\\hline
$10^1$&$10^5$&0.206&0.199&0.211&0.189&0.204&0.133       \\\hline
$10^2$&$10^3$&0.15&0.15&0.15&0.15&0.151&0.15            \\\hline
$10^2$&$10^4$&2.122&1.795&2.122&1.797&2.121&1.608       \\\hline
$10^2$&$10^5$&13.939&6.817&12.898&6.771&13.997&6.771    \\\hline
$10^3$&$10^5$&213.399&137.701&213.357&138.243&213.378&124.176         \\\hline
$10^1$&$10^6$&2.103&2.003&2.104&2.002&2.035&1.914       \\\hline
$10^2$&$10^6$&39.585&27.914&39.581&28.097&37.319&25.135 \\\hline
\end{tabular}

%% file: focus_ILS_tab2.tex
\begin{tabular}{|c|c|c|}
\hline
LB Quality&\multicolumn{2}{c|}{$\frac{2}{\pi}arctan(x)$}	\\\hline
\multicolumn{3}{|c|}{N=$10^3$, M=$10^5$}	\\\hline
$\mu$&$\sigma$&ILS           \\\hline
-0.05  &0.10&2.13         \\\hline
-0.05  &0.25&3.46         \\\hline
-0.05  &0.70&8.76         \\\hline
0.00  &0.10&2.06        \\\hline
0.00  &0.25&5.02         \\\hline
0.00  &0.70&13.67         \\\hline
+0.05  &0.10&3.66         \\\hline
+0.05  &0.25&5.33         \\\hline
+0.05  &0.70&22.42         \\\hline
\end{tabular}

%% file: convbnb_verysmall_ub_f2_tab_4.tex
\begin{tabular}{|c|c|c|c|c|c|}
	\hline
	\multicolumn{4}{|c|}{UB Quality}&\multicolumn{2}{c|}{CST, $\;\;G:x\mapsto\frac{x}{1+x}$}\\
	\hline
	N&M&B\&B&\Bar&\Cou&\Sci     \\\hline
	3&$10^2$&<0.01&<0.01&<0.01&<0.01    \\\hline
	4&$10^2$&<0.01&<0.01&<0.01&<0.01    \\\hline
	5&$10^2$&<0.01&<0.01&<0.01&<0.01    \\\hline
	6&$10^2$&<0.01&39.04&0.65&132.57        \\\hline
	7&$10^2$&<0.01&353.33&40.01&187.67      \\\hline
	8&$10^2$&<0.01&468.85&82.32&370.19      \\\hline
	9&$10^2$&0.19&409.88&134.94&423.27   \\\hline
	10&$10^2$&1.23&499.63&168.65&468.99  \\\hline
	\hline
	3&$10^3$&<0.01&<0.01&<0.01&<0.01    \\\hline
	4&$10^3$&<0.01&5.26&<0.01&<0.01           \\\hline
	5&$10^3$&<0.01&182.74&2.27&157.66       \\\hline
	6&$10^3$&0.02&511.09&48.9&283.96      \\\hline
	7&$10^3$&0.3&627.4&93.5&486.21        \\\hline
	8&$10^3$&1.01&638.12&170.31&545.56   \\\hline
	9&$10^3$&2.69&636.84&166.73&626.17   \\\hline
	10&$10^3$&10.55&642.51&207.56&629.58 \\\hline
	\hline
	3&$10^4$&<0.01&<0.01&<0.01&<0.01    \\\hline
	4&$10^4$&<0.01&6.69&<0.01&1.92    \\\hline
	5&$10^4$&<0.01&568.61&8.11&302.81              \\\hline
	6&$10^4$&0.09&665.92&68.73&607.2      \\\hline
	7&$10^4$&0.44&663.29&106.38&631.84     \\\hline
	8&$10^4$&1.28&655.42&169.17&655.31   \\\hline
	9&$10^4$&3.4&652.46&167.95&653.12    \\\hline
	10&$10^4$&13.58&648.3&211.86&644.87  \\\hline
\end{tabular}

%% file: recap_cst_largeinstances_tab4.tex
\begin{tabular}{|c|c|c|c|c|c|c|c||c|}
\hline
\multicolumn{3}{|c}{LB Quality}&\multicolumn{5}{c||}{CST, $G: x\mapsto\frac{2}{\pi}arctan(x)$}&{UB Quality}\\
\hline
N&M&BEST LB&FS&LIS&ILS&LS&TSDP&\ub2       \\\hline
$10^1$&$10^7$&TSDP (5000)&20.52&7.92&<0.01&<0.01&0&38.02     \\\hline
$10^1$&$10^8$&TSDP (5000)&20.52&7.92&<0.01&<0.01&0&38.02      \\\hline
$10^1$&$10^9$&ILS&20.52&7.92&0&<0.01&Memory&38.02           \\\hline
$10^2$&$10^6$&TSDP (5000)&25.01&2.14&<0.01&0.01&0&364.56 	 \\\hline
$10^2$&$10^7$&TSDP (5000)&25.01&2.14&<0.01&0.01&0&364.56     \\\hline
$10^2$&$10^8$&ILS&25.01&2.14&0&<0.01&Memory&364.56             \\\hline
$10^2$&$10^9$&ILS&25.01&2.14&0&<0.01&Memory&364.56 \\\hline
$10^3$&$10^6$&TSDP (500)&25.49&0.23&0.01&0.07&0&558.72   \\\hline
$10^3$&$10^7$&ILS&25.49&0.23&0&0.07&Memory&558.72             \\\hline
$10^3$&$10^8$&ILS&25.49&0.23&0&0.07&Memory&558.72            \\\hline
$10^3$&$10^9$&ILS&25.49&0.23&0&0.07&Memory&558.72            \\\hline
\end{tabular}

%% file: recap_avg_largeInstances_tab4.tex
\begin{tabular}{|c|c|c|c|c|c|c|c||c|}
\hline
\multicolumn{3}{|c}{LB Quality vs. best LB}&\multicolumn{5}{c||}{AVG, $G: x\mapsto\frac{2}{\pi}arctan(x)$}&{UB Quality}\\
\hline
N&M&BEST LB&FS&LIS&ILS&LS&TSDP&\ub2       \\\hline
$10^1$&$10^7$&TSDP (5000)&21.77&10.59&0.71&<0.01&0&139.07          \\\hline
$10^1$&$10^8$&TSDP (5000)&21.77&10.59&0.71&<0.01&0&139.07          \\\hline
$10^1$&$10^9$&LS&21.77&10.59&0.71&0&Memory&139.07 \\\hline
$10^2$&$10^6$&TSDP (500)&28.49&7.8&3.97&<0.01&0&419.13    \\\hline
$10^2$&$10^7$&TSDP (5000)&28.49&7.8&3.97&0.01&0&419.13     \\\hline
$10^2$&$10^8$&LS&28.49&7.79&3.97&0&Memory&419.16       \\\hline
$10^2$&$10^9$&LS&28.43&7.72&4.33&0&Memory&443.43       \\\hline
$10^3$&$10^6$&TSDP (500)&30.35&7.87&7.19&0.29&0&576.06     \\\hline
$10^3$&$10^7$&LS&30.31&7.82&7.11&0&Memory&576.49       \\\hline
$10^3$&$10^8$&LS&30.29&7.79&7.08&0&Memory&576.71       \\\hline
$10^3$&$10^9$&LS&30.31&7.81&7.11&0&Memory&576.52       \\\hline
\end{tabular}

%% file: recap_conv_large_f5_tab_5.tex
\begin{tabular}{|c|c|c|c|c||c|}
\hline
\multicolumn{6}{|c|}{CST,$\;G: x\mapsto\frac{2}{\pi}arctan(x)$}\\[2pt] \hline
\multicolumn{5}{|c||}{LB Quality vs. best LB}&\multicolumn{1}{c|}{UB Quality}\\[1pt]
\hline
N&M&Method&Best LB&\Bon&\ub2           					\\\hline
$10^1$&$10^7$&TSDP(5000)&$\times$&>3h&38.02			        \\\hline
$10^1$&$10^8$&TSDP(5000)&$\times$&>3h&38.02					\\\hline
$10^1$&$10^9$&ILS&$\times$&>3h&38.02						\\\hline
$10^2$&$10^6$&TSDP(5000)&$\times$&<0.01&364.56			\\\hline
$10^2$&$10^7$&TSDP(5000)&$\times$&<0.01&364.56			\\\hline
$10^2$&$10^8$&ILS&<0.01&$\times$&364.56		      	\\\hline
$10^2$&$10^9$&ILS&<0.01&$\times$&364.56				\\\hline
$10^3$&$10^6$&TSDP(500)&$\times$&>3h&558.72			       	\\\hline
$10^3$&$10^7$&ILS&$\times$&>3h&558.72     					\\\hline
$10^3$&$10^8$&ILS&$\times$&>3h&558.72      					\\\hline
$10^3$&$10^9$&ILS&$\times$&>3h&558.72						\\\hline
\end{tabular}

%% file: numalgo2026_appendix_calibration.tex
Given a prototype function $G$, the corresponding penalty function $g$ is defined by $g:x\mapsto G(\eta\, x)$, where $\eta$ is a constant which only depends on $G$.  
Let $L$ be a given level. We define a threshold $H$ such that $g(L)=H$. Hence $\eta$ is a scaling factor used to project the segment $[0;L]$ (which bounds decision variables) on the segment $[0;H]$ (which bounds the value of the penalty on the price). Prototype functions $G$ are continuous and injective on $\mathbb{R}^+$, hence invertible and we get $\eta=\dfrac{G^{-1}(H)}{L}$.\\
We discuss the selection of parameters $L$ and $H$, which fully determine $\eta$.\\ Partial sums $S_i$ go to $M$ as $i$ goes to $N$. 
Since $\lim_{+\infty} g=1$, it is natural for the function $g$ to get closer to 1 as $S_i$ goes to $M$. We thus set $L=M$.\\
We tested different values for $H \in ]0;1[$. An $H$ close to 1 leads to a corresponding penalty function $g$ with more separating power.
By separating power, we mean the distance between a naive heuristic, such as \textbf{FS} or \textbf{LIS}, and the optimum. 
A high separating power puts a stronger emphasis on the quality of the heuristic methods we propose. A very low one leads to a flat landscape, for which it is numerically more difficult to select the best solutions. 
Vector $p$ is assumed constant during the calibration. Although we tested different values for $H$, we present tables for $H=0.75$, $H=0.99$, which are the most relevant. Lastly, we also present the results for $\eta=1$, which corresponds to the absence of calibration  $(g\equiv G)$, to give another perspective.\\\\
In the following calibration tables, \textbf{FS} is the firesale and \textbf{LIS} the linear sale. The chosen prototype functions $G$ are listed on the top row. Results show the relative difference to the optimum, obtained by exact DP, and expressed as percentages:\\
\begin{table}[width=.9\linewidth,cols=1,pos=H]
	\centering
		\input{CalibTable3.tex}
		\vspace{0.3cm}\\
		\input{CalibTable5.tex}
		\vspace{0.3cm}\\
		\input{CalibTable0.tex}
\end{table}
\textbf{Conclusions:} 
\begin{itemize}
	\item  For the function $x \mapsto \frac{x}{1+x}$ and $x\mapsto \frac{2}{\pi}\arctan(x)$, the calibration $\eta_{0.99}$ has the best separating power for \textbf{LIS}.
	\item For the function $x\mapsto 1-\frac{2}{1+\sqrt{1+x}}$, calibration $\eta_{0.75}$ is slightly better for small instances for \textbf{LIS}. This advantage seems to vanish when $N\geq 2$. 
	\item For $\eta_{0.99}$, separating power remains stable for \textbf{FS} as $N$ and $M$ increase. Regarding \textbf{LIS}, it is stable in $M$ but decreases, while more slowly than for other scale factors, in $N$.
	\item Lastly, we notice that $x \mapsto \frac{x}{1+x}$ and $x\mapsto \frac{2}{\pi}\arctan(x)$ have approximately the same profile (which is linked to the similarity of their first-order derivatives), whereas $x\mapsto 1-\frac{2}{1+\sqrt{1+x}}$ returns weaker results for \textbf{LIS}, with a separating power of about $1\%$ or below, for all instances.
\end{itemize}
\vspace{0.5cm}
In \EN, we selected $L=M$ and $H=0.99$.

%% file: CalibTable3.tex
\begin{tabular}{|c|c|c|c|c|c|c|c|}
\hline
\multicolumn{2}{|c|}{\boldmath{$\eta_{0.75}$}}&\multicolumn{2}{|c|}{$\frac{x}{1+x}$}& \multicolumn{2}{|c|}{$\frac{2}{\pi}\arctan(x)$}& \multicolumn{2}{|c|}{$1-\frac{2}{1+\sqrt{1+x}}$}
\\
\hline
N&M&FS&LIS&FS&LIS&FS&LIS
\\
\hline\hline
$10^1$&$10^2$&33.44&0.84&37.85&0.52&23.58&1.43
\\
\hline
$10^1$&$10^3$&33.45&0.84&37.86&0.53&23.58&1.43
\\
\hline
$10^1$&$10^4$&33.45&0.84&37.86&0.53&23.58&1.43
\\
\hline
$10^1$&$10^5$&33.45&0.84&37.86&0.53&23.58&1.43
\\
\hline
$10^1$&$10^6$&33.45&0.84&37.86&0.53&23.58&1.43
\\
\hline
$10^2$&$10^3$&36.65&0.09&40.9&0.05&26.77&0.24
\\
\hline
$10^2$&$10^4$&36.65&0.09&40.9&0.05&26.77&0.24
\\
\hline
$10^2$&$10^5$&36.65&0.09&40.9&0.05&26.77&0.24
\\
\hline
$10^3$&$10^5$&36.97&0.01&41.19&0.01&27.1&0.03
\\
\hline
\end{tabular}

%% file: CalibTable5.tex
\begin{tabular}{|c|c|c|c|c|c|c|c|}
\hline
\multicolumn{2}{|c|}{\boldmath{$\eta_{0.99}$}}&\multicolumn{2}{|c|}{$\frac{x}{1+x}$}& \multicolumn{2}{|c|}{$\frac{2}{\pi}\arctan(x)$}& \multicolumn{2}{|c|}{$1-\frac{2}{1+\sqrt{1+x}}$}
\\
\hline
N&M&FS&LIS&FS&LIS&FS&LIS
\\
\hline\hline
$10^1$&$10^2$&18.27&6.04&20.42&7.81&5.48&0.94
\\
\hline
$10^1$&$10^3$&18.38&6.17&20.52&7.92&5.51&0.98
\\
\hline
$10^1$&$10^4$&18.38&6.17&20.52&7.92&5.51&0.98
\\
\hline
$10^1$&$10^5$&18.38&6.17&20.52&7.92&5.51&0.98
\\
\hline
$10^1$&$10^6$&18.38&6.17&20.52&7.92&5.51&0.98
\\
\hline
$10^2$&$10^3$&22.63&1.97&25.00&2.12&7.00&0.53
\\
\hline
$10^2$&$10^4$&22.65&2.00&25.00&2.14&7.08&0.62
\\
\hline
$10^2$&$10^5$&22.65&2.00&25.01&2.14&7.08&0.62
\\
\hline
$10^3$&$10^5$&23.11&0.24&25.49&0.23&7.28&0.18
\\
\hline
\end{tabular}

%% file: CalibTable0.tex
\begin{tabular}{|c|c|c|c|c|c|c|c|}
\hline
\multicolumn{2}{|c|}{\boldmath{$\eta =1$}}&\multicolumn{2}{|c|}{$\frac{x}{1+x}$}& \multicolumn{2}{|c|}{$\frac{2}{\pi}\arctan(x)$}& \multicolumn{2}{|c|}{$1-\frac{2}{1+\sqrt{1+x}}$}
\\
\hline
N&M&FS&LIS&FS&LIS&FS&LIS
\\
\hline\hline
$10^1$&$10^2$&18.16&6.03&15.25&6.41&23.91&1.98
\\
\hline
$10^1$&$10^3$&3.44&1.79&2.53&1.46&18.43&2.59
\\
\hline
$10^1$&$10^4$&0.45&0.28&0.32&0.21&9.47&1.6
\\
\hline
$10^1$&$10^5$&0.05&0.04&0.04&0.02&3.68&0.66
\\
\hline
$10^1$&$10^6$&0.01&$<\epsilon$&$<\epsilon$&$<\epsilon$&1.25&0.23
\\
\hline
$10^2$&$10^3$&4.61&1.18&3.38&1.08&22.06&0.96
\\
\hline
$10^2$&$10^4$&0.68&0.3&0.47&0.23&11.9&0.89
\\
\hline
$10^2$&$10^5$&0.09&0.05&0.06&0.04&4.79&0.45
\\
\hline
$10^3$&$10^5$&0.09&0.03&0.06&0.02&4.92&0.14
\\
\hline
\end{tabular}

%% file: numalgo2025_appendix_notations.tex
\begin{itemize}
	\item\textbf{Coverage ratio}: number of instances solved to optimality over the total number of instances.
	\item\textbf{DP:} dynamic programming.
	\item\textbf{FS:} \textit{Firesale}.The whole block is liquidated in one shot at the first time step, $x=\left(M,0,\cdots,0\right)$
	\item\textbf{ILS:} Iterative Local Search.
	\item\textbf{LIS:} \textit{Linear sale}. All the decision variables take the same value $x=\left(M/N,\cdots,M/N\right)$. Hence, the quantity to liquidate decrease linearly with time.
	\item\textbf{LS:} LocalSolver
	\item\textbf{TSDP:} two-step dynamic programming.
	\item\textbf{CST:} instance with a constant price vector $p$ and coefficients $c$. 
	\item\textbf{AVG:} average over instances with varying price vectors $p$ and coefficients $c$ (cf. Section \ref{sectionNEFrame}). 
\end{itemize}